\documentclass[11pt,a4paper]{article}

\usepackage[utf8]{inputenc}
\usepackage[T1]{fontenc}
\usepackage{lmodern}
\usepackage{amsmath,amssymb,amsthm}
\usepackage{mathrsfs}
\usepackage{mathtools}
\usepackage{bm}
\usepackage{graphicx}
\usepackage{hyperref}
\hypersetup{hidelinks}
\usepackage{enumitem}
\usepackage{geometry}
\usepackage{tikz}
\usetikzlibrary{arrows.meta}
\usepackage{tikz}
\usetikzlibrary{arrows.meta}
\usepackage{float}
\usepackage{pgfplots}
\pgfplotsset{compat=1.16}

\newtheorem{theorem}{Theorem}[section]
\newtheorem{proposition}[theorem]{Proposition}
\newtheorem{lemma}[theorem]{Lemma}
\newtheorem{corollary}[theorem]{Corollary}

\newtheorem{definition}[theorem]{Definition}
\newtheorem{remark}[theorem]{Remark}

\newtheorem{openproblem}{Open Problem}[section]

\newcommand{\R}{\mathbb{R}}
\newcommand{\N}{\mathbb{N}}

\title{\bfseries Global Convergence Regimes of the Return Dynamics in the Class $\mathcal{O}_C$}
\author{M. El Morsalani\thanks{QWave Consult, Germany. \texttt{Mohamed.elmorsalani@qwave-consult.eu}}\\
	M. Barkatou\thanks{ISTM Laboratory, Chouaib Doukkali University, Morocco. \texttt{barkatou.m@ucd.ac.ma}}}
\date{September 2026}

\begin{document}
\maketitle

\begin{abstract}
	We study the global dynamics of the geometrically generated return
	map associated with domains in the class $\mathcal O_C$.  Starting
	from a convex core $C\subset\mathbb R^N$ and a positive thickness
	function $d:\partial C\to(0,\infty)$, the return construction defines
	a discrete self-map
	\[
	F:\partial C\longrightarrow\partial C
	\]
	through an outward normal displacement to the outer boundary followed
	by the inward normal first return to the core.
	
	Building on the local geometric theory of the return map, we derive
	a uniform global one-step estimate
	\[
	F(c)-c
	=
	B_c\nabla_{\partial C}d(c)
	+
	O\!\left(|\nabla_{\partial C}d(c)|^2\right),
	\qquad
	B_c=d(c)\bigl(I-d(c)S_c\bigr)^{-1},
	\]
	under explicit smoothness, transversality, and small-slope
	assumptions.  The tangential remainder is shown to be of cubic
	order.  This yields a second-order expansion of the thickness
	increment governed by the self-adjoint preconditioned Hessian
	\[
	K_c=B_c^{1/2}H_cB_c^{1/2}.
	\]
	A directional positivity condition, and in particular the sufficient
	spectral bound
	\[
	\lambda_{\min}(K_c)\geq -2+\delta,
	\]
	implies strict global thickness ascent,
	\[
	d(F(c))-d(c)
	\geq
	\gamma |\nabla_{\partial C}d(c)|^2.
	\]
	
	Consequently,
	\[
	\mathcal V(c)=d_{\max}-d(c)
	\]
	is a strict Lyapunov function.  We prove finite Lyapunov dissipation,
	square summability of the thickness gradient, vanishing return-step
	size, and, when the critical points of $d$ are isolated, convergence
	of every orbit to a unique critical point.  At nondegenerate
	attracting maxima we obtain exponential convergence rates from the
	exact linearization
	\[
	DF(c_*)=I+B_*H_*.
	\]
	Within the strict-ascent regime, nontrivial periodic and recurrent
	orbits are excluded.
	
	Finally, numerical computations performed with the exact unit-circle
	return formula show that the spectral value $-2$ also marks the local
	loss of attraction. In a profile for which a multiplier has crossed
	$-1$, the exact first-return branch remains globally real, negative
	thickness increments are detected, and a stable nontrivial
	period-two orbit is numerically resolved. This provides a concrete
	exact-map example of the transition from curvature-preconditioned
	global convergence toward periodic return dynamics.
\end{abstract}



\tableofcontents
\section{Introduction}
\label{sec:introduction}

A geometric construction can generate a dynamical system without
any externally prescribed evolution law.  The return map studied in
this paper provides a particularly direct instance of this
phenomenon.  A point on the boundary of a convex core is first
transported along the outward normal to an enclosing boundary and
is then returned to the core by following the inward normal of that
outer boundary.  Iterating this geometric round trip produces a
discrete dynamical system on the boundary of the convex core.

The underlying setting originates in the class
$\mathcal O_C$ introduced by Barkatou
\cite{barkatou2002}.  Let
\[
C\subset\R^N
\]
be a compact convex body and let
\[
d:\partial C\longrightarrow(0,\infty)
\]
denote the thickness separating the core from the relevant outer
boundary.  The radial map is
\[
\Phi(c)=c+d(c)\nu(c),
\]
where $\nu$ is the outward unit normal to $\partial C$.  From
$\Phi(c)$ one follows the inward normal of the outer boundary to its
first intersection with the convex core.  If $\pi$ denotes this
reciprocal first-intersection map, the composition
\[
F=\pi\circ\Phi:\partial C\longrightarrow\partial C
\]
defines the return dynamics
\[
c_{k+1}=F(c_k).
\]

The local geometry of this construction was established in
\cite{barkatou2026return}.  With the sign convention
\[
S_c=-D\nu(c)|_{T_c\partial C},
\]
the return displacement near a critical point of the thickness
satisfies
\begin{equation}
F(c)-c
=
B_c\nabla_{\partial C}d(c)
+
O\!\left(
|\nabla_{\partial C}d(c)|^2
\right),
\label{eq:intro-local-expansion}
\end{equation}
where
\begin{equation}
B_c
=
d(c)\bigl(I-d(c)S_c\bigr)^{-1}.
\label{eq:intro-B}
\end{equation}
For a convex core the operator $B_c$ is self-adjoint and positive
definite.  In a principal direction with curvature $\kappa_i(c)$,
its eigenvalue is
\[
\frac{d(c)}
{1+d(c)\kappa_i(c)}.
\]
The leading return displacement is therefore not an ordinary
gradient step: the thickness gradient is modified anisotropically
by the curvature of the underlying core.

This observation places the return mechanism naturally beside
modern developments in geometry-aware optimization.  Riemannian
formulations of shape optimization interpret boundary updates,
shape Hessians, and gradient flows through the geometry of spaces
of shapes \cite{schulz2014riemannian,welker2021suitable,
	plotnikov2023geometric}.  More generally, optimization constrained
to smooth manifolds replaces Euclidean update directions by
intrinsic or geometry-dependent directions; recent developments
include Riemannian proximal-gradient methods and their global and
local convergence theory
\cite{huangwei2022rpg,huangwei2023inexact}.
The analogy with the present construction is structural rather than
algorithmic.  Here the preconditioner $B_c$ is not selected to
accelerate an optimization procedure.  It is forced by the normal
geometry of the two boundaries.

Indeed, since $B_c$ is positive definite, one may introduce on
$T_c\partial C$ the metric
\begin{equation}
\langle u,v\rangle_{B,c}
:=
\langle B_c^{-1}u,v\rangle.
\label{eq:intro-induced-metric}
\end{equation}
With respect to this metric, the vector
\[
B_c\nabla_{\partial C}d(c)
\]
is precisely the Riemannian gradient of the thickness.  Thus the
return geometry itself generates both a scalar landscape $d$ and
the metric that preconditions its leading dynamics.  This
geometry-induced coupling between landscape, metric, and discrete
evolution is one of the motivations for studying the return map
globally.

The local theory alone, however, does not determine the behavior of
its iterates.  In particular,
\eqref{eq:intro-local-expansion} is a local estimate near critical
points and does not imply
\[
d(F(c))>d(c)
\]
throughout $\partial C$.  Nor does the fact that the leading
direction is positive guarantee a global Lyapunov function: a
finite discrete step contains second-order and genuinely nonlinear
geometric effects.  Questions of global monotonicity, orbit
convergence, recurrence, and nontrivial periodic dynamics were
therefore left open in \cite{barkatou2026return}.

The purpose of the present paper is to develop this global theory.
The first step is to replace the local expansion
\eqref{eq:intro-local-expansion} by a uniform estimate on the entire
compact phase space.  Under explicit smoothness, return
transversality, and small-slope assumptions, we prove
\begin{equation}
F(c)-c
=
B_c\nabla_{\partial C}d(c)+R(c),
\qquad
|R(c)|
\leq
C|\nabla_{\partial C}d(c)|^2,
\label{eq:intro-global-estimate}
\end{equation}
uniformly in $c$.  More importantly, the tangential component of
the remainder is of cubic order.  This stronger tangential estimate
makes it possible to retain the full second-order information in the
variation of the thickness.

Let $H_c$ denote the Riemannian Hessian operator of $d$ and define
the self-adjoint preconditioned Hessian
\begin{equation}
K_c
=
B_c^{1/2}H_cB_c^{1/2}.
\label{eq:intro-K}
\end{equation}
Writing
\[
z_c
=
B_c^{1/2}\nabla_{\partial C}d(c),
\]
we obtain the uniform second-order structure
\begin{equation}
d(F(c))-d(c)
=
\left\langle
z_c,
\left(
I+\frac12K_c
\right)z_c
\right\rangle
+
\mathcal E(c),
\label{eq:intro-thickness-increment}
\end{equation}
where
\[
|\mathcal E(c)|
\leq
\eta\!\left(|\nabla_{\partial C}d(c)|\right)
|\nabla_{\partial C}d(c)|^2,
\qquad
\eta(s)\to0\quad(s\downarrow0).
\]
This formula identifies the mechanism by which the local
preconditioned-ascent law can lead to genuine global monotonicity.

A natural sufficient condition is the uniform spectral inequality
\begin{equation}
\lambda_{\min}(K_c)
\geq
-2+\delta,
\qquad
\delta>0,
\label{eq:intro-minus-two}
\end{equation}
although the global theory is formulated more generally through a
directional positivity condition along $z_c$.  Under the resulting
ascent assumptions we prove the quantitative inequality
\begin{equation}
d(F(c))-d(c)
\geq
\gamma
|\nabla_{\partial C}d(c)|^2,
\qquad
\gamma>0,
\label{eq:intro-strict-ascent}
\end{equation}
away from the fixed-point set.  Consequently,
\begin{equation}
\mathcal V(c)
=
d_{\max}-d(c)
\label{eq:intro-Lyapunov}
\end{equation}
is a strict Lyapunov function for the return dynamics.

This part of the argument is related in spirit to a broad class of
discrete convergence mechanisms in which sufficient decrease,
gradient control, and compactness convert a scalar Lyapunov estimate
into asymptotic information about the iterates.  Such principles
appear prominently in modern convergence theory for nonconvex
descent methods
\cite{attouchboltesvaiter2013,bento2025plk}, although the present
map differs fundamentally from an algorithmic descent scheme:
neither the step nor its preconditioner is chosen externally.

From \eqref{eq:intro-strict-ascent} we obtain the exact telescoping
bound
\[
\sum_{k=0}^{\infty}
|\nabla_{\partial C}d(c_k)|^2
<\infty,
\]
and hence
\[
|\nabla_{\partial C}d(c_k)|
\longrightarrow0,
\qquad
|c_{k+1}-c_k|
\longrightarrow0.
\]
If the critical points of $d$ are isolated, compactness and the
vanishing step size then prevent the orbit from moving indefinitely
between distinct critical neighbourhoods.  Every orbit therefore
converges to a unique critical point.  No Morse nondegeneracy
assumption is needed for this global conclusion.

This distinction is useful in relation to the
Kurdyka--\L ojasiewicz theory.  KL inequalities provide a powerful
framework for proving finite-length properties, convergence to
single critical points, and quantitative rates for broad classes of
gradient and proximal systems
\cite{kurdyka1998,attouchboltesvaiter2013,
	huangwei2022rpg,huangwei2023inexact}.
They are not assumed in the main convergence theorem of the present
paper.  Here isolation of the critical points is sufficient.
KL-type hypotheses become relevant instead when one asks how the
argument might be extended to nonisolated or degenerate critical
sets, a question discussed in Section~\ref{sec:discussion}.

The local and global analyses are linked by the same spectral
quantity.  At a critical point $c_*$, the exact linearization
derived in \cite{barkatou2026return} is
\begin{equation}
DF(c_*)
=
I+B_*H_*.
\label{eq:intro-linearization}
\end{equation}
Since $B_*H_*$ is similar to
\[
K_*=B_*^{1/2}H_*B_*^{1/2},
\]
if $\mu_j$ denotes an eigenvalue of $K_*$, the corresponding
multiplier is
\begin{equation}
\lambda_j=1+\mu_j.
\label{eq:intro-multiplier}
\end{equation}
At a nondegenerate maximum, attraction therefore requires
\[
-2<\mu_j<0.
\]
The value $-2$ is exactly the point at which a multiplier crosses
$-1$.  On the other hand,
\eqref{eq:intro-thickness-increment} shows that the same value is
the natural spectral boundary at which the full quadratic ascent form
\[
I+\frac12K_c
\]
ceases to be positive definite.  Thus the appearance of $-2$ at both the local
and global levels is a geometric consequence of the return
construction rather than an independently imposed stability
constant.

The unit-circle geometry makes this transition explicit.  Because
the first-intersection return can be evaluated exactly, the spectral
mechanism can be examined without replacing the return map by its
first-order approximation.  In exact-map examples lying on the
$K>-2$ side of the spectral threshold, numerical diagnostics detect
no negative thickness increments and show convergence toward
attracting maxima.  These computations illustrate behavior compatible
with the analytical ascent regime, but they are not presented as
certified instances of all quantitative hypotheses of
Theorem~\ref{thm:uniform-strict-ascent}, in particular the global
small-slope threshold.  As the relevant preconditioned eigenvalue
approaches $-2$, convergence may become alternating while remaining
stable.  We then examine a profile for which the multiplier crosses
$-1$ while the exact first-return branch remains globally real.
Numerical evaluation of the exact map detects negative thickness
increments and resolves a stable nontrivial period-two orbit.  These
computations provide exact-map evidence that violation of the
sufficient convergence regime can produce genuinely different
dynamics; they are not used to infer chaotic behavior.

The geometric structure investigated here is also relevant beyond
the specific return problem.  Normal boundary variations, shape
gradients, geometric regularization, and gradient flows are basic
tools in modern shape optimization
\cite{schulz2014riemannian,welker2021suitable,
	plotnikov2023geometric}.  Such methods arise in problems involving
moving interfaces, inclusion identification, solid mechanics,
fluid mechanics, and PDE-constrained design.  The present results
do not turn the $\mathcal O_C$ return map into a general shape
optimization algorithm.  They instead identify a mechanism that can
be relevant whenever repeated normal interactions between nested
boundaries generate an effective boundary dynamics: curvature can
precondition the motion, alter stability thresholds, and determine
whether a scalar geometric quantity remains monotone.

A broader structural analogy also occurs in geometry-aware
optimization on nonlinear parameter spaces.  Riemannian methods on
matrix manifolds appear in machine learning
\cite{huangwei2022rpg,han2024riemannian} and in quantum-information
and quantum-control problems \cite{luchnikov2021quantum}.  We do not
claim that the present return map models those systems directly; the
common principle is only that intrinsic geometry modifies the
effective gradient direction and hence the stability of the induced
dynamics.  Here that modification is generated explicitly by boundary
curvature through \eqref{eq:intro-B}.

The main contributions of the paper can therefore be summarized as
follows.  We establish a uniform global version of the local
return-map expansion, including a cubic tangential remainder; derive
the second-order thickness increment and identify the
curvature-preconditioned Hessian $K_c$; obtain explicit directional
and spectral sufficient conditions for strict global thickness
ascent; construct the Lyapunov function
$\mathcal V=d_{\max}-d$ and prove finite dissipation, gradient
square summability, and asymptotic regularity; prove convergence of
every orbit to a unique critical point under the weaker hypothesis
that the critical points are isolated; determine local exponential
rates from the exact linearization; describe convergence classes and
attraction basins without conflating them; exclude nontrivial
periodic and recurrent dynamics in the strict-ascent regime; and,
finally, use numerical computations of the exact unit-circle map to
document loss of monotonicity and a numerically resolved stable
period-two orbit beyond the $-2$ threshold.

The paper is organized as follows.
Section~\ref{sec:geometric-framework} recalls only the geometric and
local dynamical results from \cite{barkatou2026return} required in
the sequel.
Section~\ref{sec:uniform-global-estimate} derives the uniform
one-step estimate.
Section~\ref{sec:preconditioned-ascent} develops the second-order
thickness expansion and the global ascent conditions.
Section~\ref{sec:lyapunov-structure} establishes the Lyapunov
structure and finite dissipation.
Section~\ref{sec:global-convergence} proves convergence to a unique
critical point.
Section~\ref{sec:local-rates} studies local asymptotics and
convergence rates.
Section~\ref{sec:basins-recurrence} analyzes convergence classes,
attraction basins, and recurrence.
Section~\ref{sec:circle-dynamics} gives the exact unit-circle
experiments, including the numerically resolved stable period-two
orbit beyond the monotone regime.
Finally, Section~\ref{sec:discussion} discusses the scope of the
results and several open problems arising from the boundary between
the convergent and nonconvergent regimes.

\section{Geometric Framework and Structural Inputs}
\label{sec:geometric-framework}

The return map considered in this paper is the geometrically generated
self-map of the convex-core boundary introduced and analyzed in
\cite{barkatou2026return}. We recall only the structural ingredients
needed for the global analysis and impose explicitly the additional
regularity required below. In particular, the local expansion proved in
\cite{barkatou2026return} will not be used as a global estimate without
further justification.

\subsection{Convex core, thickness, and return map}
\label{subsec:core-return-map}

Let
\[
C\subset\R^N,
\qquad N\geq 2,
\]
be a compact convex set with nonempty interior. We assume throughout
this paper that
\[
\partial C\in C^3.
\]
The outward unit normal to $\partial C$ is denoted by
\[
\nu:\partial C\longrightarrow \mathbb S^{N-1}.
\]

We use the sign convention
\begin{equation}
S_c
:=
-D\nu(c)\big|_{T_c\partial C},
\label{eq:shape-operator}
\end{equation}
for the shape operator of the convex core. Thus, if
\[
\kappa_1(c),\ldots,\kappa_{N-1}(c)\geq 0
\]
are the usual principal curvatures of the convex hypersurface
$\partial C$, then the eigenvalues of $S_c$ are
\[
-\kappa_1(c),\ldots,-\kappa_{N-1}(c).
\]

Let $\Omega$ belong to the class $\mathcal O_C$ considered in
\cite{barkatou2002,barkatou2026return}.  For completeness, we recall
the defining geometry in the present smooth-core setting.  Write
\[
N_\Omega:=\{x\in\partial\Omega: n(x)\text{ exists}\},
\qquad
\ell_c:=\{c+r\nu(c):r\in\R\},
\qquad
\Delta_c^+:=\{c+r\nu(c):r\geq0\}.
\]
The $C$-geometric normal property requires
\begin{enumerate}[label=(P\arabic*),leftmargin=2.2em]
\item $\operatorname{int}(C)\subset\Omega$;
\item $\partial\Omega$ is Lipschitz outside $C$;
\item for every $x\in N_\Omega\setminus C$, the inward normal
half-line $\{x+t n(x):t\geq0\}$ intersects $C$;
\item for every $c\in\partial C$, the set $\ell_c\cap\Omega$ is
connected.
\end{enumerate}
The class $\mathcal O_C$ consists of such domains for which, in
addition, $\partial\Omega\setminus C$ is an orientable
$C^{1,1}$ hypersurface, every outward half-normal $\Delta_c^+$ meets
$\partial\Omega\setminus C$, and this intersection is transverse.
These are precisely the structural assumptions imported from
\cite{barkatou2026return}; the stronger $C^2$ regularity of the
thickness and the uniform inward-return transversality used below are
additional hypotheses of the present global analysis.

The geometric normal property and the radial accessibility result
proved in \cite{barkatou2026return} imply that every point
$c\in\partial C$ determines a unique point of the outer boundary along
the outward normal ray. Hence there exists a positive thickness
function
\[
d:\partial C\longrightarrow(0,\infty)
\]
such that the radial map is
\begin{equation}
\Phi(c)
=
c+d(c)\nu(c).
\label{eq:radial-map}
\end{equation}

For the global dynamical analysis we impose the additional regularity
assumption
\begin{equation}
d\in C^2(\partial C).
\label{eq:thickness-regularity}
\end{equation}
Since $\partial C$ is compact, there exist constants
\begin{equation}
0<d_{\min}
\leq d(c)
\leq d_{\max}<\infty,
\qquad c\in\partial C.
\label{eq:thickness-bounds}
\end{equation}

Starting from
\[
x=\Phi(c)\in\partial\Omega\setminus C,
\]
one follows the inward unit normal $n(x)$ to its first intersection
with the convex core. Denoting this first-intersection map by
\[
\pi:\partial\Omega\setminus C\longrightarrow\partial C,
\]
the return map is
\begin{equation}
F
=
\pi\circ\Phi:
\partial C\longrightarrow\partial C.
\label{eq:return-map}
\end{equation}
The discrete return dynamics is therefore
\begin{equation}
c_{k+1}=F(c_k),
\qquad k\geq0.
\label{eq:return-iteration}
\end{equation}

In addition to the radial transversality contained in the definition
of $\mathcal O_C$, we impose the following \emph{uniform inward-return
transversality} hypothesis: there exists a constant $\tau_0>0$ such that
\begin{equation}
-\big\langle
n(\Phi(c)),
\nu(F(c))
\big\rangle
\geq\tau_0,
\qquad c\in\partial C.
\label{eq:uniform-return-transversality}
\end{equation}
In particular,
\[
\big\langle n(\Phi(c)),\nu(F(c))\big\rangle<0
\]
for every $c\in\partial C$. By the reciprocal-map regularity result of
\cite{barkatou2026return}, this uniform non-tangency implies local
regularity of the first-hit branch along $\Phi(\partial C)$; in
particular, the induced return map $F$ is continuous on $\partial C$.
The uniform hypothesis prevents tangential first returns and supplies
the regularity needed below without deriving it circularly from the
continuity of $F$.

\begin{figure}[H]
	\centering
	\resizebox{0.92\columnwidth}{!}{%
		\begin{tikzpicture}[
		>=Latex,
		line cap=round,
		line join=round,
		every node/.style={font=\small}
		]
		
		\draw[thick] (0.4,0.1) ellipse [x radius=3.9, y radius=2.2];
		
		\draw[thick] (0,0) ellipse [x radius=1.8, y radius=1.02];
		
		\node at (-1.0,-1.38) {$\partial C$};
		\node at (2.45,-2.08) {$\partial\Omega\setminus C$};
		
		\coordinate (c) at (1.48,0.50);
		\coordinate (x) at (3.35,1.18);
		\coordinate (cp) at (0.18,0.98);
		
		\fill (c) circle (1.5pt);
		\fill (x) circle (1.5pt);
		\fill (cp) circle (1.5pt);
		
		\node[below right] at (c) {$c$};
		\node[above right] at (x) {$x=\Phi(c)$};
		\node[above left] at (cp) {$c'=\pi(x)=F(c)$};
		
		\draw[->,thick] (c) -- (x)
		node[pos=0.52, above=6pt] {$\Phi$};
		
		\draw[->,thick] (x) -- (cp)
		node[pos=0.45, above=6pt] {$\pi$};
		
		\draw[->,dashed,thick] (c) to[bend left=36]
		node[pos=0.56, left=6pt] {$F=\pi\circ\Phi$} (cp);
		
		\draw[->] (x) -- ++(-0.95,-0.28)
		node[below] {$n(x)$};
		
		\node at (0.50,2.00) {$c'-x=t(x)n(x)$};
		
		\end{tikzpicture}%
	}
	\caption{Reciprocal map and return map. Starting from
		$x=\Phi(c)\in\partial\Omega\setminus C$, the reciprocal map follows
		the inward normal $n(x)$ to its first intersection
		$c'=\pi(x)\in\partial C$. The induced return map on the convex core is
		$F=\pi\circ\Phi$.}
	\label{fig:roundtrip-maps}
\end{figure}
\subsection{Exact one-step geometry}
\label{subsec:exact-one-step-geometry}

We next recall the exact geometric identities established in
\cite{barkatou2026return}. Set
\begin{equation}
g_c
:=
\nabla_{\partial C}d(c)
\in T_c\partial C
\label{eq:thickness-gradient}
\end{equation}
and introduce the tangential operator
\begin{equation}
L_c
:=
I-d(c)S_c.
\label{eq:L-operator}
\end{equation}
By the sign convention \eqref{eq:shape-operator}, the eigenvalues of
$L_c$ in the principal directions are
\begin{equation}
1+d(c)\kappa_i(c)>0,
\qquad i=1,\ldots,N-1.
\label{eq:L-eigenvalues}
\end{equation}
Consequently, $L_c$ is positive definite and invertible for every
$c\in\partial C$.

Define
\begin{equation}
w_c
:=
L_c^{-1}g_c
=
\bigl(I-d(c)S_c\bigr)^{-1}
\nabla_{\partial C}d(c).
\label{eq:tangential-slope}
\end{equation}
The exact inward normal to the outer boundary at the radial point
$\Phi(c)$ is
\begin{equation}
n(\Phi(c))
=
\frac{-\nu(c)+w_c}
{\sqrt{1+|w_c|^2}}.
\label{eq:exact-inward-normal}
\end{equation}

Let
\[
t_c:=t(\Phi(c))
\]
be the first-return distance from $\Phi(c)$ to the convex core and set
\begin{equation}
\sigma_c
:=
\frac{t_c}{\sqrt{1+|w_c|^2}}.
\label{eq:scaled-return-parameter}
\end{equation}
Then the return displacement admits the exact decomposition
\begin{equation}
F(c)-c
=
\sigma_c w_c
+
\bigl(d(c)-\sigma_c\bigr)\nu(c).
\label{eq:exact-return-decomposition}
\end{equation}
In particular, The metric projection onto the tangent space $T_c\partial C$ is given by
\begin{equation}
P_{T_c\partial C}\bigl(F(c)-c\bigr)
=
\sigma_c w_c,
\label{eq:exact-tangential-displacement}
\end{equation}
whereas
\begin{equation}
\big\langle F(c)-c,\nu(c)\big\rangle
=
d(c)-\sigma_c.
\label{eq:exact-normal-displacement}
\end{equation}

Equations
\eqref{eq:exact-inward-normal}--\eqref{eq:exact-normal-displacement}
are exact. No small-gradient or asymptotic assumption has been used.
They will form the starting point for the uniform estimate derived in
Section~\ref{sec:uniform-global-estimate}.

\subsection{The curvature-preconditioning operator}
\label{subsec:preconditioning-operator}

The principal operator governing the local return dynamics is
\begin{equation}
B_c
:=
d(c)L_c^{-1}
=
d(c)\bigl(I-d(c)S_c\bigr)^{-1}.
\label{eq:B-operator}
\end{equation}
In a principal direction $e_i(c)$ of the convex core,
\begin{equation}
B_c e_i(c)
=
\frac{d(c)}
{1+d(c)\kappa_i(c)}
e_i(c).
\label{eq:B-principal-eigenvalues}
\end{equation}
Thus $B_c$ is self-adjoint and positive definite.

Let
\begin{equation}
\kappa_{\max}
:=
\max_{\substack{c\in\partial C\\1\leq i\leq N-1}}
\kappa_i(c).
\label{eq:kappa-max}
\end{equation}
Compactness of $\partial C$ and
\eqref{eq:thickness-bounds} imply the uniform spectral bounds
\begin{equation}
\beta_- |v|^2
\leq
\langle B_c v,v\rangle
\leq
\beta_+ |v|^2,
\qquad
v\in T_c\partial C,
\label{eq:B-uniform-coercivity}
\end{equation}
where one may take
\begin{equation}
\beta_-
:=
\frac{d_{\min}}
{1+d_{\max}\kappa_{\max}},
\qquad
\beta_+
:=
d_{\max}.
\label{eq:B-uniform-bounds}
\end{equation}
In particular, the positive direction selected by $B_c$ is uniformly
nondegenerate on the compact phase space.

Let
\[
H_c:T_c\partial C\longrightarrow T_c\partial C
\]
denote the self-adjoint Hessian operator associated with the
Riemannian Hessian of $d$:
\begin{equation}
\langle H_c v,z\rangle
=
\operatorname{Hess}_{\partial C}d(c)[v,z].
\label{eq:hessian-operator}
\end{equation}
We shall repeatedly use the self-adjoint
curvature-preconditioned Hessian
\begin{equation}
K_c
:=
B_c^{1/2}H_cB_c^{1/2}.
\label{eq:preconditioned-hessian}
\end{equation}
Unlike $B_cH_c$, the operator $K_c$ is self-adjoint, while
$B_cH_c$ and $K_c$ are similar and therefore have the same
real spectrum.

This operator will play two related roles below. Locally, at a
critical point, its eigenvalues determine the multipliers of the
linearized return map. Globally, it will enter the second-order
variation of the thickness along one return step.

\subsection{Fixed points and local linearization}
\label{subsec:imported-local-dynamics}

The foundational analysis in \cite{barkatou2026return} identifies the
equilibria of the return dynamics with the critical points of the
thickness function. Under the present smoothness assumptions this
holds pointwise:
\begin{equation}
F(c)=c
\quad\Longleftrightarrow\quad
\nabla_{\partial C}d(c)=0.
\label{eq:fixed-critical-equivalence}
\end{equation}

Let $c_*\in\partial C$ be a critical point and write
\begin{equation}
d_*:=d(c_*),
\qquad
B_*:=d_*\bigl(I-d_*S_*\bigr)^{-1},
\qquad
H_*:=H_{c_*}.
\label{eq:critical-point-operators}
\end{equation}
The exact derivative of the return map at $c_*$ is
\begin{equation}
DF(c_*)\big|_{T_{c_*}\partial C}
=
I+B_*H_*.
\label{eq:critical-linearization}
\end{equation}
If
\[
\mu_1,\ldots,\mu_{N-1}
\]
are the eigenvalues of
\[
K_*
=
B_*^{1/2}H_*B_*^{1/2},
\]
then the multipliers of the return map are
\begin{equation}
\lambda_j
=
1+\mu_j,
\qquad
j=1,\ldots,N-1.
\label{eq:critical-multipliers}
\end{equation}

In particular, the threshold
\begin{equation}
-2<\mu_j<0
\label{eq:local-attraction-threshold}
\end{equation}
characterizes linear attraction at a nondegenerate local maximum.
Nondegenerate minima are repelling, while a nondegenerate saddle has
at least one unstable direction. These local facts will serve as a
consistency requirement for the global theory developed below.

\subsection{Local versus global control}
\label{subsec:local-versus-global}

A central point for the present paper is the distinction between the
local result already available and the global estimate that remains
to be proved.

Near a critical point, \cite{barkatou2026return} establishes the
quadratic remainder estimate
\begin{equation}
F(c)-c
=
B_c\nabla_{\partial C}d(c)
+
R(c),
\qquad
|R(c)|
\leq
\Lambda
|\nabla_{\partial C}d(c)|^2,
\label{eq:imported-local-remainder}
\end{equation}
for $c$ in a sufficiently small local-return neighbourhood. Hence
the leading displacement points in a curvature-preconditioned
gradient-ascent direction.

Estimate \eqref{eq:imported-local-remainder}, however, is a
\emph{local} statement. It does not imply that
\begin{equation}
d(F(c))>d(c)
\label{eq:desired-global-increase}
\end{equation}
for every noncritical point of $\partial C$, nor does it provide a
uniform quadratic remainder bound on the whole phase space.

Accordingly, we shall not assume
\eqref{eq:imported-local-remainder} globally. The next section derives
a uniform one-step estimate directly from the exact geometry
\eqref{eq:exact-return-decomposition}, under explicit global
small-slope and local-return hypotheses. That estimate will provide
the bridge from the local theory of
\cite{barkatou2026return} to the Lyapunov and convergence results of
the present paper.

\section{Uniform Global Control of One Return Step}
\label{sec:uniform-global-estimate}

The quadratic remainder estimate established in
\cite{barkatou2026return} is local: it is proved in a sufficiently
small return neighbourhood of a critical point of the thickness
function. For the global convergence analysis developed below, this
local statement must be replaced by a uniform estimate valid on the
whole compact phase space.

The purpose of this section is to derive such an estimate directly
from the exact return geometry recalled in
Section~\ref{sec:geometric-framework}. The essential ingredients are
the compactness of $\partial C$, uniform local graph representations
of the convex core, global return transversality, and a uniform
small-slope condition.

\subsection{Uniform normal-graph geometry of the convex core}
\label{subsec:uniform-normal-graphs}

We begin with a uniform version of the local normal-graph
representation of $\partial C$.

\begin{lemma}[Uniform normal-graph representation]
	\label{lem:uniform-normal-graph}
	There exist constants
	\[
	\rho_0>0,
	\qquad
	r_0>0,
	\qquad
	K_C>0,
	\]
	such that for every $c\in\partial C$, the portion of $\partial C$
	contained in $B_{\rho_0}(c)$ can be represented uniquely in the form
	\begin{equation}
	u
	\longmapsto
	c+u+h_c(u)\nu(c),
	\qquad
	u\in B_{r_0}^{T_c\partial C}(0),
	\label{eq:uniform-normal-graph}
	\end{equation}
	where
	\begin{equation}
	h_c(0)=0,
	\qquad
	Dh_c(0)=0,
	\label{eq:uniform-graph-origin}
	\end{equation}
	and
	\begin{equation}
	|h_c(u)|
	\leq
	K_C|u|^2.
	\label{eq:uniform-graph-quadratic-bound}
	\end{equation}
	Moreover, convexity implies
	\begin{equation}
	h_c(u)\leq0.
	\label{eq:uniform-graph-convex-sign}
	\end{equation}
	The constants $\rho_0$, $r_0$, and $K_C$ may be chosen
	independently of $c$.
\end{lemma}

\begin{proof}
	Since $\partial C$ is a compact $C^3$ hypersurface, every
	$c\in\partial C$ admits a local representation as a graph over its
	tangent space. The $C^2$ norms of these local graph functions can be
	chosen uniformly after passing to a finite covering of $\partial C$.
	This yields uniform constants $\rho_0$, $r_0$, and $K_C$ such that
	\eqref{eq:uniform-normal-graph}--%
	\eqref{eq:uniform-graph-quadratic-bound} hold.
	
	To obtain the sign, observe that the tangent hyperplane at $c$ is a
	supporting hyperplane of the convex set $C$. Hence every nearby point
	\[
	q=c+u+h_c(u)\nu(c)\in\partial C
	\]
	satisfies
	\[
	\langle q-c,\nu(c)\rangle
	=
	h_c(u)
	\leq0,
	\]
	which proves \eqref{eq:uniform-graph-convex-sign}.
\end{proof}

\subsection{Uniform control of the return distance}
\label{subsec:uniform-return-distance}

Recall that
\[
t_c:=t(\Phi(c))
\]
denotes the first-return distance from the outer boundary to the
convex core and
\[
\sigma_c
=
\frac{t_c}{\sqrt{1+|w_c|^2}}.
\]

\begin{lemma}[Uniform boundedness of the return parameter]
\label{lem:uniform-return-parameter}
There exists a constant $T_0<\infty$ such that
\begin{equation}
0<t_c\leq T_0,
\qquad
0<\sigma_c\leq T_0,
\qquad
c\in\partial C.
\label{eq:uniform-return-parameter-bound}
\end{equation}
One may take
\begin{equation}
T_0:=d_{\max}+\operatorname{diam}(C).
\label{eq:T0-geometric-choice}
\end{equation}
\end{lemma}

\begin{proof}
Since $F(c)\in\partial C$ and
$\Phi(c)=c+d(c)\nu(c)$,
\[
t_c
=
|\Phi(c)-F(c)|
\leq
|\Phi(c)-c|+|c-F(c)|
\leq
d_{\max}+\operatorname{diam}(C).
\]
This gives the asserted bound for $t_c$. The estimate for
$\sigma_c$ follows from
\[
\sigma_c
=
\frac{t_c}{\sqrt{1+|w_c|^2}}
\leq
t_c.
\]
\end{proof}

The next lemma is the step that allows the local graph geometry of
Lemma~\ref{lem:uniform-normal-graph} to be used uniformly.

\begin{lemma}[Small slope implies uniform local return]
	\label{lem:small-slope-local-return}
	There exists a constant
	\[
	\varepsilon_{\mathrm{lr}}>0
	\]
	such that
	\begin{equation}
	|w_c|
	\leq
	\varepsilon_{\mathrm{lr}}
	\label{eq:local-return-slope-threshold}
	\end{equation}
	implies
	\begin{equation}
	|F(c)-c|<\rho_0.
	\label{eq:uniform-local-return}
	\end{equation}
	Here $\rho_0$ is the uniform graph radius from
	Lemma~\ref{lem:uniform-normal-graph}.
\end{lemma}

\begin{proof}
	Suppose that the assertion were false. Then there would exist a
	sequence $(c_j)\subset\partial C$ such that
	\begin{equation}
	|w_{c_j}|\longrightarrow0,
	\qquad
	|F(c_j)-c_j|\geq\rho_0.
	\label{eq:local-return-contradiction-sequence}
	\end{equation}
	By compactness, after passing to a subsequence,
	\[
	c_j\longrightarrow c_*
	\in\partial C.
	\]
	
	The map
	\[
	c\longmapsto
	w_c
	=
	\bigl(I-d(c)S_c\bigr)^{-1}
	\nabla_{\partial C}d(c)
	\]
	is continuous under the standing regularity assumptions. Hence
	\[
	w_{c_*}=0.
	\]
	Since $I-d(c_*)S_{c_*}$ is invertible, this is equivalent to
	\[
	\nabla_{\partial C}d(c_*)=0.
	\]
	The fixed-point characterization established in
	\cite{barkatou2026return} therefore gives
	\[
	F(c_*)=c_*.
	\]
	
	By the continuity of $F$ stated after
	\eqref{eq:uniform-return-transversality},
	\[
	F(c_j)-c_j
	\longrightarrow
	F(c_*)-c_*
	=
	0,
	\]
	contradicting
	\eqref{eq:local-return-contradiction-sequence}.
\end{proof}

\begin{remark}
	\label{rem:small-gradient-sufficient}
	Since the eigenvalues of
	\[
	L_c=I-d(c)S_c
	\]
	are at least $1$, one has
	\[
	\|L_c^{-1}\|\leq1.
	\]
	Therefore
	\begin{equation}
	|w_c|
	\leq
	|\nabla_{\partial C}d(c)|.
	\label{eq:slope-gradient-upper-bound}
	\end{equation}
	Thus the stronger but simpler condition
	\[
	\|\nabla_{\partial C}d\|_{L^\infty(\partial C)}
	\leq
	\varepsilon_{\mathrm{lr}}
	\]
	is sufficient for the uniform local-return property.
\end{remark}

\subsection{Uniform quadratic remainder estimate}
\label{subsec:uniform-quadratic-remainder}

Set
\begin{equation}
L_{\max}:=1+d_{\max}\kappa_{\max}.
\label{eq:L-max}
\end{equation}
Since the eigenvalues of $L_c=I-d(c)S_c$ belong to
$[1,L_{\max}]$, one has
\[
|\nabla_{\partial C}d(c)|
\leq
L_{\max}|w_c|.
\]

We now impose the global small-slope condition
\begin{equation}
\|w\|_{L^\infty(\partial C)}
\leq
\varepsilon_0,
\qquad
0<\varepsilon_0
\leq
\min\left\{
1,\,
\varepsilon_{\mathrm{lr}},\,
\frac{r_0}{2T_0L_{\max}}
\right\}.
\label{eq:global-small-slope}
\end{equation}
The last restriction guarantees
\begin{equation}
T_0|\nabla_{\partial C}d(c)|
\leq
\frac{r_0}{2}
<r_0
\qquad
\text{for every }c\in\partial C,
\label{eq:taylor-domain-uniform}
\end{equation}
which will keep the Taylor remainder of Section~\ref{subsec:uniform-thickness-taylor}
inside the uniform normal-graph chart.

Under this condition, every return point $F(c)$ lies in the uniform
normal-graph neighbourhood of its starting point $c$.

\begin{theorem}[Uniform global one-step estimate]
	\label{thm:uniform-global-one-step}
	Assume \eqref{eq:global-small-slope}. Then there exists a constant
	$\Lambda_0>0$, independent of $c$, such that for every
	$c\in\partial C$,
	\begin{equation}
	F(c)-c
	=
	B_c\nabla_{\partial C}d(c)
	+
	R(c),
	\label{eq:uniform-global-expansion}
	\end{equation}
	where
	\begin{equation}
	|R(c)|
	\leq
	\Lambda_0
	|\nabla_{\partial C}d(c)|^2.
	\label{eq:uniform-global-remainder}
	\end{equation}
	
	More precisely,
	\begin{equation}
	0
	\leq
	\sigma_c-d(c)
	\leq
	K_CT_0^2|w_c|^2,
	\label{eq:sigma-minus-d-global}
	\end{equation}
	and the tangential component satisfies the stronger estimate
	\begin{equation}
	P_{T_c\partial C}\bigl(F(c)-c\bigr)
	=
	B_c\nabla_{\partial C}d(c)
	+
	R_T(c),
	\label{eq:uniform-tangential-expansion}
	\end{equation}
	with
	\begin{equation}
	|R_T(c)|
	\leq
	K_CT_0^2|w_c|^3.
	\label{eq:tangential-cubic-remainder}
	\end{equation}
\end{theorem}

\begin{proof}
	Fix $c\in\partial C$. By
	Lemma~\ref{lem:small-slope-local-return} and
	\eqref{eq:global-small-slope}, the point $F(c)$ belongs to the
	uniform normal-graph neighbourhood based at $c$. Hence there exists a
	unique $u_c\in T_c\partial C$ such that
	\begin{equation}
	F(c)-c
	=
	u_c+h_c(u_c)\nu(c).
	\label{eq:return-in-uniform-graph}
	\end{equation}
	
	On the other hand, the exact return decomposition recalled in
	\eqref{eq:exact-return-decomposition} gives
	\[
	F(c)-c
	=
	\sigma_c w_c
	+
	\bigl(d(c)-\sigma_c\bigr)\nu(c).
	\]
	Comparing tangent and normal components yields
	\begin{equation}
	u_c
	=
	\sigma_cw_c
	\label{eq:uniform-graph-tangent}
	\end{equation}
	and
	\begin{equation}
	d(c)-\sigma_c
	=
	h_c(\sigma_cw_c).
	\label{eq:uniform-global-closure}
	\end{equation}
	
	By convexity,
	\[
	h_c(\sigma_cw_c)\leq0,
	\]
	so
	\[
	\sigma_c\geq d(c).
	\]
	Furthermore, using
	\eqref{eq:uniform-graph-quadratic-bound} and
	\eqref{eq:uniform-return-parameter-bound},
	\[
	0
	\leq
	\sigma_c-d(c)
	\leq
	K_C\sigma_c^2|w_c|^2
	\leq
	K_CT_0^2|w_c|^2,
	\]
	which proves \eqref{eq:sigma-minus-d-global}.
	
	We now subtract the leading tangential term $d(c)w_c$ from the exact
	return displacement:
	\begin{align}
	F(c)-c-d(c)w_c
	&=
	(\sigma_c-d(c))w_c
	+
	(d(c)-\sigma_c)\nu(c)
	\nonumber\\
	&=
	(\sigma_c-d(c))
	\bigl(w_c-\nu(c)\bigr).
	\label{eq:global-remainder-identity}
	\end{align}
	Therefore,
	\begin{align}
	|F(c)-c-d(c)w_c|
	&=
	|\sigma_c-d(c)|
	\sqrt{1+|w_c|^2}
	\nonumber\\
	&\leq
	K_CT_0^2
	\sqrt{1+\varepsilon_0^2}
	\,|w_c|^2.
	\label{eq:global-remainder-w-bound}
	\end{align}
	Since $\varepsilon_0\leq1$, one may take
	\[
	\Lambda_w
	=
	\sqrt{2}\,K_CT_0^2.
	\]
	
	Finally,
	\[
	d(c)w_c
	=
	d(c)
	\bigl(I-d(c)S_c\bigr)^{-1}
	\nabla_{\partial C}d(c)
	=
	B_c\nabla_{\partial C}d(c),
	\]
	and \eqref{eq:slope-gradient-upper-bound} gives
	\[
	|w_c|
	\leq
	|\nabla_{\partial C}d(c)|.
	\]
	Thus
	\[
	|R(c)|
	\leq
	\Lambda_w
	|\nabla_{\partial C}d(c)|^2,
	\]
	which proves
	\eqref{eq:uniform-global-expansion}--%
	\eqref{eq:uniform-global-remainder}.
	
	For the tangential component,
	\[
	P_{T_c\partial C}(F(c)-c)
	=
	\sigma_cw_c.
	\]
	Hence
	\[
	R_T(c)
	=
	(\sigma_c-d(c))w_c,
	\]
	and \eqref{eq:sigma-minus-d-global} gives
	\[
	|R_T(c)|
	\leq
	K_CT_0^2|w_c|^3.
	\]
	This proves
	\eqref{eq:uniform-tangential-expansion} and
	\eqref{eq:tangential-cubic-remainder}.
\end{proof}

\subsection{Gradient--step comparability}
\label{subsec:gradient-step-comparability}

The preceding theorem gives several uniform estimates that will be
used repeatedly in the global convergence argument.

Recall from \eqref{eq:L-max} that
\[
L_{\max}=1+d_{\max}\kappa_{\max}.
\]
Since the eigenvalues of $L_c$ belong to
$[1,L_{\max}]$, the eigenvalues of $L_c^{-1}$ belong to
$[L_{\max}^{-1},1]$. Therefore
\begin{equation}
\frac{1}{L_{\max}}
|\nabla_{\partial C}d(c)|
\leq
|w_c|
\leq
|\nabla_{\partial C}d(c)|.
\label{eq:gradient-slope-comparability}
\end{equation}

\begin{corollary}[Uniform control of one-step displacement]
	\label{cor:uniform-step-control}
	Under the assumptions of
	Theorem~\ref{thm:uniform-global-one-step}, there exists a constant
	$C_F>0$ such that
	\begin{equation}
	|F(c)-c|
	\leq
	C_F
	|\nabla_{\partial C}d(c)|
	\label{eq:step-gradient-upper-bound}
	\end{equation}
	for every $c\in\partial C$.
	
	Moreover, the tangential displacement satisfies
	\begin{equation}
	\beta_-
	|\nabla_{\partial C}d(c)|
	\leq
	\left|
	P_{T_c\partial C}(F(c)-c)
	\right|
	\leq
	T_0
	|\nabla_{\partial C}d(c)|,
	\label{eq:tangential-step-gradient-comparability}
	\end{equation}
	where
	\begin{equation}
	\beta_-
	=
	\frac{d_{\min}}
	{1+d_{\max}\kappa_{\max}}
	>0.
	\label{eq:beta-minus-recalled}
	\end{equation}
\end{corollary}

\begin{proof}
	From Theorem~\ref{thm:uniform-global-one-step},
	\[
	|F(c)-c|
	\leq
	d_{\max}|w_c|
	+
	\Lambda_w|w_c|^2.
	\]
	Since $|w_c|\leq\varepsilon_0$,
	\[
	|F(c)-c|
	\leq
	\bigl(d_{\max}+\Lambda_w\varepsilon_0\bigr)
	|w_c|.
	\]
	Using
	\eqref{eq:slope-gradient-upper-bound} proves
	\eqref{eq:step-gradient-upper-bound}.
	
	For the tangential component,
	\[
	\left|
	P_{T_c\partial C}(F(c)-c)
	\right|
	=
	\sigma_c|w_c|.
	\]
	Since
	\[
	d(c)\leq\sigma_c\leq T_0,
	\]
	and by
	\eqref{eq:gradient-slope-comparability},
	\[
	\sigma_c|w_c|
	\geq
	d_{\min}
	\frac{1}{L_{\max}}
	|\nabla_{\partial C}d(c)|
	=
	\beta_-
	|\nabla_{\partial C}d(c)|.
	\]
	The upper bound follows from
	\[
	\sigma_c|w_c|
	\leq
	T_0|\nabla_{\partial C}d(c)|.
	\]
\end{proof}

\begin{remark}
	\label{rem:global-estimate-vs-local-paper}
	Theorem~\ref{thm:uniform-global-one-step} is not a restatement of the
	local expansion proved in \cite{barkatou2026return}. The new content
	is the uniformity over the entire compact phase space
	$\partial C$, obtained from the global transversality and
	small-slope assumptions.
	
	A further feature of the proof is the distinction between the full
	ambient remainder and its tangential component:
	\[
	R(c)=O(|\nabla_{\partial C}d(c)|^2),
	\]
	whereas
	\[
	R_T(c)=O(|\nabla_{\partial C}d(c)|^3).
	\]
	The latter estimate will be important in the second-order expansion
	of the thickness increment in the next section.
\end{remark}

\begin{remark}[Role of the small-slope hypothesis]
\label{rem:role-small-slope}
The small-slope condition is a sufficient geometric hypothesis used
to place every first return inside a uniform normal-graph
neighbourhood and thereby obtain the uniform quadratic one-step
estimate.  No necessity is claimed.  In particular, the exact return
map may remain globally well defined outside this perturbative regime,
as the unit-circle computations of
Section~\ref{sec:circle-dynamics} illustrate.
\end{remark}

\section{Curvature-Preconditioned Ascent}
\label{sec:preconditioned-ascent}

The uniform one-step estimate of
Section~\ref{sec:uniform-global-estimate} shows that the tangential
part of the return displacement is, to leading order,
\[
B_c\nabla_{\partial C}d(c),
\]
where
\[
B_c
=
d(c)\bigl(I-d(c)S_c\bigr)^{-1}.
\]
Since $B_c$ is positive definite, this leading displacement points
in a direction of increasing thickness. The purpose of the present
section is to determine conditions under which this local tendency
produces a genuine one-step increase of the thickness on the whole
phase space.

The second-order term is essential. It introduces the
curvature-preconditioned Hessian
\[
K_c
=
B_c^{1/2}H_cB_c^{1/2},
\]
and leads naturally to the same spectral threshold $-2$ that appears
in the local stability theory of the return map.

\subsection{Uniform second-order expansion of the thickness}
\label{subsec:uniform-thickness-taylor}

For $c\in\partial C$, let
\[
\Psi_c(u)
=
c+u+h_c(u)\nu(c)
\]
be the uniform normal-graph parametrization supplied by
Lemma~\ref{lem:uniform-normal-graph}. Define
\begin{equation}
\widetilde d_c(u)
:=
d(\Psi_c(u)).
\label{eq:pulled-back-thickness}
\end{equation}

The following uniform Taylor formula will be used throughout this
section.

\begin{lemma}[Uniform second-order Taylor formula]
	\label{lem:uniform-thickness-taylor}
	There exists a nondecreasing function
	\[
	\omega_d:[0,r_0]\longrightarrow[0,\infty)
	\]
	with
	\begin{equation}
	\omega_d(r)\longrightarrow0
	\qquad\text{as }r\downarrow0,
	\label{eq:hessian-modulus-limit}
	\end{equation}
	such that, for every $c\in\partial C$ and every
	$u\in T_c\partial C$ with $|u|<r_0$,
	\begin{align}
	d(\Psi_c(u))-d(c)
	&=
	\big\langle g_c,u\big\rangle
	+
	\frac12
	\big\langle H_cu,u\big\rangle
	+
	\mathcal T_c(u),
	\label{eq:uniform-second-order-taylor}
	\end{align}
	where
	\begin{equation}
	|\mathcal T_c(u)|
	\leq
	\frac12
	\omega_d(|u|)|u|^2.
	\label{eq:uniform-taylor-remainder}
	\end{equation}
	Here
	\[
	g_c=\nabla_{\partial C}d(c)
	\]
	and $H_c$ is the Hessian operator defined in
	\eqref{eq:hessian-operator}.
\end{lemma}

\begin{proof}
	At $u=0$,
	\[
	D\Psi_c(0)=I
	\]
	on $T_c\partial C$. Hence
	\[
	D\widetilde d_c(0)[u]
	=
	\langle g_c,u\rangle.
	\]
	
	Moreover,
	\[
	D^2\Psi_c(0)[u,v]
	=
	D^2h_c(0)[u,v]\nu(c)
	\]
	is normal to $T_c\partial C$. Thus the Christoffel symbols of the
	coordinates induced by $\Psi_c$ vanish at the origin. Consequently,
	the ordinary Hessian of $\widetilde d_c$ at $u=0$ coincides with the
	intrinsic Hessian of $d$:
	\begin{equation}
	D^2\widetilde d_c(0)[u,v]
	=
	\operatorname{Hess}_{\partial C}d(c)[u,v]
	=
	\langle H_cu,v\rangle.
	\label{eq:chart-hessian-intrinsic}
	\end{equation}
	
	Since $d\in C^2(\partial C)$ and $\partial C$ is compact, choose
	a finite atlas of the uniform normal-graph neighbourhoods above.
	In this atlas the graph parametrizations depend $C^2$-continuously
	on the base point $c$. Consequently, on the associated compact
	closed disk bundle, the map
	\[
	(c,u)\longmapsto D^2(d\circ\Psi_c)(u)
	\]
	is uniformly continuous. Hence the second derivatives of the
	pulled-back functions $\widetilde d_c$ admit a single modulus of
	continuity independent of $c$. Therefore there exists a modulus
	$\omega_d$ satisfying \eqref{eq:hessian-modulus-limit} such that
	\[
	\left\|
	D^2\widetilde d_c(v)
	-
	D^2\widetilde d_c(0)
	\right\|
	\leq
	\omega_d(|u|)
	\]
	whenever $|v|\leq|u|<r_0$.
	
	Taylor's formula with integral remainder then yields
	\eqref{eq:uniform-second-order-taylor} and
	\eqref{eq:uniform-taylor-remainder}.
\end{proof}

\subsection{Second-order variation along one return step}
\label{subsec:second-order-return-variation}

We now apply Lemma~\ref{lem:uniform-thickness-taylor} to the return
point $F(c)$.

Let
\begin{equation}
u_c
:=
P_{T_c\partial C}\bigl(F(c)-c\bigr).
\label{eq:return-tangent-coordinate}
\end{equation}
By Theorem~\ref{thm:uniform-global-one-step},
\begin{equation}
u_c
=
B_cg_c+r_c,
\label{eq:tangent-coordinate-decomposition}
\end{equation}
where
\begin{equation}
|r_c|
\leq
C_T|g_c|^3,
\qquad
C_T:=K_CT_0^2.
\label{eq:tangent-coordinate-cubic}
\end{equation}
Moreover,
\begin{equation}
|u_c|
=
\sigma_c|w_c|
\leq
T_0|g_c|.
\label{eq:tangent-coordinate-bound}
\end{equation}
By \eqref{eq:taylor-domain-uniform},
\[
T_0|g_c|<r_0,
\]
so every occurrence of $\omega_d(T_0|g_c|)$ below lies inside the
domain of the modulus introduced in
Lemma~\ref{lem:uniform-thickness-taylor}.

Define
\begin{equation}
z_c
:=
B_c^{1/2}g_c.
\label{eq:z-definition}
\end{equation}

\begin{theorem}[Second-order thickness increment]
	\label{thm:second-order-thickness-increment}
	Under the assumptions of
	Theorem~\ref{thm:uniform-global-one-step}, one has
	\begin{equation}
	d(F(c))-d(c)
	=
	\left\langle
	z_c,
	\left(
	I+\frac12K_c
	\right)z_c
	\right\rangle
	+
	\mathcal E(c),
	\label{eq:second-order-thickness-increment}
	\end{equation}
	where
	\[
	K_c
	=
	B_c^{1/2}H_cB_c^{1/2},
	\]
	and the remainder satisfies
	\begin{equation}
	|\mathcal E(c)|
	\leq
	\eta(|g_c|)|g_c|^2.
	\label{eq:thickness-increment-remainder}
	\end{equation}
	Here $\eta:[0,L_{\max}\varepsilon_0]\to[0,\infty)$ may be chosen so that
	\begin{equation}
	\eta(s)\longrightarrow0
	\qquad\text{as }s\downarrow0.
	\label{eq:eta-limit}
	\end{equation}
	More explicitly, one may take
	\begin{align}
	\eta(s)
	={}&
	C_Ts^2
	+
	M_Hd_{\max}C_Ts^2
	+
	\frac12M_HC_T^2s^4
	\nonumber\\
	&\quad
	+
	\frac12T_0^2\omega_d(T_0s),
	\label{eq:eta-explicit}
	\end{align}
	where
	\begin{equation}
	M_H
	:=
	\max_{c\in\partial C}\|H_c\|.
	\label{eq:hessian-uniform-bound}
	\end{equation}
\end{theorem}

\begin{proof}
	Since
	\[
	F(c)=\Psi_c(u_c),
	\]
	Lemma~\ref{lem:uniform-thickness-taylor} gives
	\begin{align}
	d(F(c))-d(c)
	&=
	\langle g_c,u_c\rangle
	+
	\frac12
	\langle H_cu_c,u_c\rangle
	+
	\mathcal T_c(u_c).
	\label{eq:increment-before-substitution}
	\end{align}
	
	Set
	\[
	q_c:=B_cg_c.
	\]
	Then
	\[
	u_c=q_c+r_c.
	\]
	
	For the linear term,
	\begin{align}
	\langle g_c,u_c\rangle
	&=
	\langle g_c,B_cg_c\rangle
	+
	\langle g_c,r_c\rangle
	\nonumber\\
	&=
	|z_c|^2
	+
	\langle g_c,r_c\rangle,
	\label{eq:linear-increment-term}
	\end{align}
	because
	\[
	|z_c|^2
	=
	\langle g_c,B_cg_c\rangle.
	\]
	Furthermore,
	\begin{equation}
	|\langle g_c,r_c\rangle|
	\leq
	C_T|g_c|^4.
	\label{eq:linear-error-bound}
	\end{equation}
	
	For the quadratic term,
	\begin{align}
	\langle H_cu_c,u_c\rangle
	={}&
	\langle H_cq_c,q_c\rangle
	+
	2\langle H_cq_c,r_c\rangle
	+
	\langle H_cr_c,r_c\rangle.
	\label{eq:hessian-term-decomposition}
	\end{align}
	Since
	\[
	q_c
	=
	B_cg_c
	=
	B_c^{1/2}z_c,
	\]
	one obtains
	\begin{align}
	\langle H_cq_c,q_c\rangle
	&=
	\left\langle
	H_cB_c^{1/2}z_c,
	B_c^{1/2}z_c
	\right\rangle
	\nonumber\\
	&=
	\langle K_cz_c,z_c\rangle.
	\label{eq:hessian-preconditioned-form}
	\end{align}
	
	Using
	\[
	\|B_c\|\leq d_{\max}
	\]
	and \eqref{eq:tangent-coordinate-cubic},
	\begin{align}
	\left|
	\langle H_cq_c,r_c\rangle
	\right|
	&\leq
	M_Hd_{\max}C_T|g_c|^4,
	\label{eq:hessian-cross-bound}
	\\
	\left|
	\langle H_cr_c,r_c\rangle
	\right|
	&\leq
	M_HC_T^2|g_c|^6.
	\label{eq:hessian-remainder-bound}
	\end{align}
	
	Finally, by
	\eqref{eq:tangent-coordinate-bound} and
	\eqref{eq:uniform-taylor-remainder},
	\begin{align}
	|\mathcal T_c(u_c)|
	&\leq
	\frac12
	\omega_d(|u_c|)|u_c|^2
	\nonumber\\
	&\leq
	\frac12
	T_0^2
	\omega_d(T_0|g_c|)
	|g_c|^2.
	\label{eq:taylor-return-bound}
	\end{align}
	
	Substituting
	\eqref{eq:linear-increment-term}--%
	\eqref{eq:taylor-return-bound}
	into
	\eqref{eq:increment-before-substitution}
	gives
	\eqref{eq:second-order-thickness-increment}.
	Collecting the remainder terms yields
	\eqref{eq:thickness-increment-remainder} with
	\eqref{eq:eta-explicit}. Since
	\[
	\omega_d(r)\to0
	\qquad\text{as }r\to0,
	\]
	one has $\eta(s)\to0$ as $s\to0$.
\end{proof}

\begin{remark}
	\label{rem:second-order-structure}
	Formula \eqref{eq:second-order-thickness-increment} identifies the
	two contributions governing one return step:
	\[
	|z_c|^2
	\]
	is the positive first-order ascent term, whereas
	\[
	\frac12\langle K_cz_c,z_c\rangle
	\]
	is the second-order correction induced jointly by the Hessian of the
	thickness and the curvature-dependent preconditioner $B_c$.
\end{remark}

\subsection{A uniform ascent condition}
\label{subsec:uniform-ascent-condition}

The preceding theorem suggests the following condition.

\begin{definition}[Uniform preconditioned-ascent condition]
	\label{def:uniform-ascent-condition}
	We say that the thickness landscape satisfies the uniform
	preconditioned-ascent condition if there exists a constant
	$\alpha_0>0$ such that
	\begin{equation}
	\left\langle
	z_c,
	\left(
	I+\frac12K_c
	\right)z_c
	\right\rangle
	\geq
	\alpha_0|z_c|^2
	\label{eq:directional-ascent-condition}
	\end{equation}
	for every $c\in\partial C$.
\end{definition}

Condition \eqref{eq:directional-ascent-condition} is directional: it
requires positivity only in the dynamically relevant direction
\[
z_c=B_c^{1/2}\nabla_{\partial C}d(c).
\]
A stronger but simpler spectral condition gives an immediate
sufficient criterion.

\begin{proposition}[Spectral sufficient condition]
	\label{prop:spectral-ascent-condition}
	Assume that there exists $\delta>0$ such that
	\begin{equation}
	\lambda_{\min}(K_c)
	\geq
	-2+\delta
	\qquad
	\text{for every }c\in\partial C.
	\label{eq:global-minus-two-condition}
	\end{equation}
	Then the uniform preconditioned-ascent condition holds with
	\begin{equation}
	\alpha_0=\frac{\delta}{2}.
	\label{eq:alpha-from-delta}
	\end{equation}
\end{proposition}

\begin{proof}
	Condition \eqref{eq:global-minus-two-condition} implies
	\[
	K_c
	\geq
	(-2+\delta)I
	\]
	in the sense of quadratic forms. Therefore
	\[
	I+\frac12K_c
	\geq
	\frac{\delta}{2}I.
	\]
	Hence
	\[
	\left\langle
	z_c,
	\left(
	I+\frac12K_c
	\right)z_c
	\right\rangle
	\geq
	\frac{\delta}{2}|z_c|^2,
	\]
	which proves the assertion.
\end{proof}

\begin{remark}
	\label{rem:directional-vs-spectral-condition}
	The directional condition
	\eqref{eq:directional-ascent-condition} is weaker than the uniform
	spectral condition \eqref{eq:global-minus-two-condition}. The global
	convergence theory below therefore does not require every direction
	of $K_c$ to satisfy the bound $-2+\delta$; it is enough that the
	quadratic form be uniformly positive along the preconditioned
	gradient direction.
	
	The spectral condition is nevertheless geometrically transparent and
	will be useful for comparison with the local stability threshold.
	It controls directions that may never be sampled by the return
	dynamics and is therefore generally stronger than necessary. In
	particular, $\lambda_{\min}(K_c)<-2$ at some point does not by itself
	imply loss of thickness monotonicity, because the corresponding
	eigendirection may be poorly aligned with $z_c$.
\end{remark}

\subsection{Strict global increase of the thickness}
\label{subsec:strict-thickness-ascent}

Recall the uniform coercivity estimate
\begin{equation}
|z_c|^2
=
\langle B_cg_c,g_c\rangle
\geq
\beta_-|g_c|^2,
\label{eq:z-gradient-coercivity}
\end{equation}
where
\[
\beta_-
=
\frac{d_{\min}}
{1+d_{\max}\kappa_{\max}}
>0.
\]

Also recall from
\eqref{eq:gradient-slope-comparability} that
\begin{equation}
|g_c|
\leq
L_{\max}|w_c|,
\qquad
L_{\max}
=
1+d_{\max}\kappa_{\max}.
\label{eq:gradient-from-slope}
\end{equation}

\begin{theorem}[Uniform strict ascent]
	\label{thm:uniform-strict-ascent}
	Assume the uniform preconditioned-ascent condition
	\eqref{eq:directional-ascent-condition}. Then there exists
	\[
	\varepsilon_{\mathrm{asc}}>0
	\]
	such that, if
	\begin{equation}
	\|w\|_{L^\infty(\partial C)}
	\leq
	\varepsilon_{\mathrm{asc}},
	\label{eq:ascent-small-slope}
	\end{equation}
	the return map satisfies
	\begin{equation}
	d(F(c))-d(c)
	\geq
	\gamma
	|\nabla_{\partial C}d(c)|^2
	\label{eq:strict-thickness-ascent}
	\end{equation}
	for every $c\in\partial C$, where
	\begin{equation}
	\gamma
	:=
	\frac{\alpha_0\beta_-}{2}
	>0.
	\label{eq:ascent-gamma}
	\end{equation}
	
	Consequently,
	\begin{equation}
	d(F(c))\geq d(c)
	\label{eq:thickness-nondecreasing}
	\end{equation}
	for every $c\in\partial C$, with equality if and only if
	\begin{equation}
	\nabla_{\partial C}d(c)=0,
	\label{eq:equality-critical}
	\end{equation}
	or equivalently,
	\begin{equation}
	F(c)=c.
	\label{eq:equality-fixed}
	\end{equation}
\end{theorem}

\begin{proof}
	By Theorem~\ref{thm:second-order-thickness-increment},
	\begin{align}
	d(F(c))-d(c)
	&=
	\left\langle
	z_c,
	\left(
	I+\frac12K_c
	\right)z_c
	\right\rangle
	+
	\mathcal E(c).
	\end{align}
	The ascent condition gives
	\[
	\left\langle
	z_c,
	\left(
	I+\frac12K_c
	\right)z_c
	\right\rangle
	\geq
	\alpha_0|z_c|^2
	\geq
	\alpha_0\beta_-|g_c|^2.
	\]
	Hence
	\begin{equation}
	d(F(c))-d(c)
	\geq
	\left(
	\alpha_0\beta_-
	-
	\eta(|g_c|)
	\right)
	|g_c|^2.
	\label{eq:ascent-before-smallness}
	\end{equation}
	
	Since $\eta(s)\to0$ as $s\downarrow0$, choose
	$s_{\mathrm{asc}}>0$ so that
	\begin{equation}
	\eta(s)
	\leq
	\frac{\alpha_0\beta_-}{2}
	\qquad
	\text{for }0\leq s\leq s_{\mathrm{asc}}.
	\label{eq:eta-absorption}
	\end{equation}
	Now choose
	\begin{equation}
	\varepsilon_{\mathrm{asc}}
	\leq
	\min
	\left\{
	\varepsilon_0,
	\frac{s_{\mathrm{asc}}}{L_{\max}}
	\right\}.
	\label{eq:epsilon-ascent-choice}
	\end{equation}
	Then
	\eqref{eq:gradient-from-slope} gives
	\[
	|g_c|
	\leq
	s_{\mathrm{asc}}
	\]
	for every $c$, so
	\eqref{eq:ascent-before-smallness} becomes
	\[
	d(F(c))-d(c)
	\geq
	\frac{\alpha_0\beta_-}{2}
	|g_c|^2.
	\]
	This proves \eqref{eq:strict-thickness-ascent}.
	
	If $g_c\neq0$, the right-hand side is strictly positive. If
	$g_c=0$, the fixed-point characterization from
	\cite{barkatou2026return} gives $F(c)=c$, so equality holds.
	This proves the final assertions.
\end{proof}

The threshold $\varepsilon_{\mathrm{asc}}$ is quantitative once the
uniform graph constants and a modulus of continuity for the Hessian
are specified: one first chooses $s_{\mathrm{asc}}$ from
\eqref{eq:eta-absorption} and then uses
\eqref{eq:epsilon-ascent-choice}. No universal closed-form value is
expected under the present $C^2$ assumptions, because the admissible
threshold depends on these geometric and continuity data.

Under the spectral condition
\eqref{eq:global-minus-two-condition}, one may take
\[
\alpha_0=\frac{\delta}{2},
\]
and hence
\begin{equation}
\gamma
=
\frac{\delta\beta_-}{4}.
\label{eq:spectral-ascent-gamma}
\end{equation}

\subsection{Two-sided control of the thickness increment}
\label{subsec:two-sided-thickness-increment}

For later use we also record an upper bound.

Since $K_c$ depends continuously on $c$ and $\partial C$ is compact,
set
\begin{equation}
M_K
:=
\max_{c\in\partial C}\|K_c\|.
\label{eq:K-uniform-bound}
\end{equation}

\begin{corollary}[Quadratic comparability of the thickness increment]
	\label{cor:thickness-increment-comparability}
	Under the assumptions of
	Theorem~\ref{thm:uniform-strict-ascent}, there exists a constant
	$\Gamma>0$ such that
	\begin{equation}
	\gamma
	|\nabla_{\partial C}d(c)|^2
	\leq
	d(F(c))-d(c)
	\leq
	\Gamma
	|\nabla_{\partial C}d(c)|^2
	\label{eq:thickness-increment-comparability}
	\end{equation}
	for every $c\in\partial C$.
\end{corollary}

\begin{proof}
	The lower bound is
	\eqref{eq:strict-thickness-ascent}.
	
	For the upper bound, use
	\[
	|z_c|^2
	\leq
	\beta_+|g_c|^2,
	\qquad
	\beta_+=d_{\max},
	\]
	and
	\[
	\left|
	\left\langle
	z_c,
	\left(
	I+\frac12K_c
	\right)z_c
	\right\rangle
	\right|
	\leq
	\left(
	1+\frac12M_K
	\right)|z_c|^2.
	\]
	The remainder estimate
	\eqref{eq:thickness-increment-remainder} then gives
	\[
	d(F(c))-d(c)
	\leq
	\left[
	\beta_+
	\left(
	1+\frac12M_K
	\right)
	+
	\eta(s_{\mathrm{asc}})
	\right]
	|g_c|^2.
	\]
	Thus one may take
	\begin{equation}
	\Gamma
	=
	\beta_+
	\left(
	1+\frac12M_K
	\right)
	+
	\eta(s_{\mathrm{asc}}).
	\label{eq:Gamma-definition}
	\end{equation}
\end{proof}

\subsection{Relation with the local stability threshold}
\label{subsec:global-local-threshold}

The spectral condition
\[
\lambda_{\min}(K_c)>-2
\]
is not introduced arbitrarily. The same threshold occurs in the
exact local linearization established in
\cite{barkatou2026return}.

At a critical point $c_*$,
\[
DF(c_*)
=
I+B_*H_*,
\]
and $B_*H_*$ is similar to the self-adjoint operator
\[
K_*
=
B_*^{1/2}H_*B_*^{1/2}.
\]
If $\mu_j$ are the eigenvalues of $K_*$, the multipliers of the
return map are
\[
\lambda_j=1+\mu_j.
\]
Consequently, a nondegenerate local maximum is linearly attracting
precisely when
\[
-2<\mu_j<0
\]
for every $j$.

The global condition
\eqref{eq:global-minus-two-condition} controls the lower part of this
same spectral range uniformly over $\partial C$. It prevents the
second-order curvature--Hessian contribution from overwhelming the
positive first-order ascent term.

The role of the threshold is therefore consistent at both scales:
locally it prevents an attracting maximum from overshooting through
a multiplier of modulus at least one, while globally it provides a
natural sufficient condition for strict increase of the thickness
after one return step.

We emphasize that
\eqref{eq:global-minus-two-condition} is a sufficient global
condition, not a necessary characterization of all geometries for
which the thickness is monotone.

\section{Lyapunov Structure of the Return Dynamics}
\label{sec:lyapunov-structure}

Throughout this section we assume the hypotheses of
Theorem~\ref{thm:uniform-strict-ascent}. In particular, there exist
constants
\[
0<\gamma\leq\Gamma<\infty
\]
such that
\begin{equation}
\gamma
|\nabla_{\partial C}d(c)|^2
\leq
d(F(c))-d(c)
\leq
\Gamma
|\nabla_{\partial C}d(c)|^2
\label{eq:two-sided-thickness-gain}
\end{equation}
for every $c\in\partial C$.

Since $d\in C^2(\partial C)$ and $\partial C$ is compact, we set
\begin{equation}
d_{\max}
:=
\max_{c\in\partial C}d(c).
\label{eq:global-thickness-maximum}
\end{equation}

\subsection{Thickness deficit as a Lyapunov function}
\label{subsec:thickness-deficit}

The return dynamics moves, under the assumptions of the preceding
section, in the direction of increasing thickness. It is therefore
natural to measure the remaining thickness deficit from the global
maximum.

\begin{definition}[Thickness-deficit Lyapunov function]
	\label{def:thickness-deficit}
	Define
	\begin{equation}
	\mathcal V(c)
	:=
	d_{\max}-d(c),
	\qquad
	c\in\partial C.
	\label{eq:lyapunov-function}
	\end{equation}
	Then
	\[
	\mathcal V(c)\geq0
	\]
	on $\partial C$.
\end{definition}

For later use, define also the one-step thickness gain
\begin{equation}
\mathcal D(c)
:=
d(F(c))-d(c).
\label{eq:dissipation-function}
\end{equation}
Although $\mathcal D$ measures an increase of the thickness, it is
the exact dissipation of the deficit $\mathcal V$:
\begin{equation}
\mathcal V(c)-\mathcal V(F(c))
=
\mathcal D(c).
\label{eq:exact-lyapunov-balance}
\end{equation}

\begin{theorem}[Strict Lyapunov inequality]
	\label{thm:strict-lyapunov}
	For every $c\in\partial C$,
	\begin{equation}
	\gamma
	|\nabla_{\partial C}d(c)|^2
	\leq
	\mathcal V(c)-\mathcal V(F(c))
	\leq
	\Gamma
	|\nabla_{\partial C}d(c)|^2.
	\label{eq:lyapunov-two-sided}
	\end{equation}
	Equivalently,
	\begin{equation}
	-\Gamma
	|\nabla_{\partial C}d(c)|^2
	\leq
	\mathcal V(F(c))-\mathcal V(c)
	\leq
	-\gamma
	|\nabla_{\partial C}d(c)|^2.
	\label{eq:lyapunov-decrease}
	\end{equation}
	
	Consequently,
	\begin{equation}
	\mathcal V(F(c))
	\leq
	\mathcal V(c),
	\label{eq:lyapunov-monotonicity}
	\end{equation}
	with equality if and only if
	\begin{equation}
	\nabla_{\partial C}d(c)=0.
	\label{eq:lyapunov-equality-critical}
	\end{equation}
	Under the fixed-point characterization recalled in
	Section~\ref{subsec:imported-local-dynamics}, this is equivalent to
	\begin{equation}
	F(c)=c.
	\label{eq:lyapunov-equality-fixed}
	\end{equation}
\end{theorem}

\begin{proof}
	By definition,
	\[
	\mathcal V(c)-\mathcal V(F(c))
	=
	d(F(c))-d(c).
	\]
	The two-sided estimate
	\eqref{eq:lyapunov-two-sided} is therefore exactly
	\eqref{eq:two-sided-thickness-gain}.
	
	If $\nabla_{\partial C}d(c)\neq0$, the lower bound gives
	\[
	\mathcal V(F(c))
	<
	\mathcal V(c).
	\]
	If $\nabla_{\partial C}d(c)=0$, then $c$ is a fixed point of $F$,
	and equality follows.
\end{proof}

\begin{remark}
	\label{rem:lyapunov-relative-fixed-set}
	The function $\mathcal V$ is a strict Lyapunov function relative to
	the fixed-point set of the return map. Its zero set consists of the
	global maxima of $d$, whereas its one-step variation vanishes at
	every critical point of $d$, because every such point is an exact
	fixed point of $F$. Thus no claim is made that all equilibria minimize
	$\mathcal V$.
\end{remark}

\subsection{Monotonicity along return orbits}
\label{subsec:orbit-monotonicity}

Let
\begin{equation}
c_{k+1}=F(c_k),
\qquad k\geq0,
\label{eq:orbit-notation}
\end{equation}
be an arbitrary return orbit, and write
\begin{equation}
d_k:=d(c_k),
\qquad
\mathcal V_k:=\mathcal V(c_k),
\qquad
g_k:=\nabla_{\partial C}d(c_k).
\label{eq:orbit-quantities}
\end{equation}

The strict ascent theorem gives
\begin{equation}
d_{k+1}-d_k
\geq
\gamma|g_k|^2
\geq0.
\label{eq:orbit-thickness-increase}
\end{equation}
Hence $(d_k)$ is nondecreasing. Since
\[
d_k\leq d_{\max},
\]
there exists
\begin{equation}
d_\infty
\in[d(c_0),d_{\max}]
\label{eq:thickness-limit-range}
\end{equation}
such that
\begin{equation}
d_k\longrightarrow d_\infty.
\label{eq:thickness-limit}
\end{equation}
Equivalently,
\begin{equation}
\mathcal V_k
\longrightarrow
\mathcal V_\infty
:=
d_{\max}-d_\infty.
\label{eq:lyapunov-limit}
\end{equation}

At this stage, the existence of the scalar limit $d_\infty$ does not
yet imply convergence of the orbit $(c_k)$ itself. Establishing that
stronger conclusion is the subject of
Section~\ref{sec:global-convergence}.

\subsection{Exact telescoping balance and finite dissipation}
\label{subsec:finite-dissipation}

The Lyapunov identity telescopes exactly along every orbit.

\begin{proposition}[Finite Lyapunov dissipation]
	\label{prop:finite-dissipation}
	For every integer $n\geq1$,
	\begin{equation}
	\mathcal V(c_n)
	+
	\sum_{k=0}^{n-1}
	\mathcal D(c_k)
	=
	\mathcal V(c_0).
	\label{eq:finite-lyapunov-balance}
	\end{equation}
	Consequently,
	\begin{equation}
	\sum_{k=0}^{\infty}
	\mathcal D(c_k)
	=
	\mathcal V(c_0)-\mathcal V_\infty
	=
	d_\infty-d(c_0)
	<\infty.
	\label{eq:total-dissipation}
	\end{equation}
\end{proposition}

\begin{proof}
	Equation \eqref{eq:exact-lyapunov-balance} gives
	\[
	\mathcal D(c_k)
	=
	\mathcal V(c_k)-\mathcal V(c_{k+1}).
	\]
	Summing from $k=0$ to $n-1$ gives
	\[
	\sum_{k=0}^{n-1}\mathcal D(c_k)
	=
	\mathcal V(c_0)-\mathcal V(c_n),
	\]
	which is equivalent to
	\eqref{eq:finite-lyapunov-balance}.
	
	Passing to the limit $n\to\infty$ and using
	\eqref{eq:lyapunov-limit} yields
	\eqref{eq:total-dissipation}.
\end{proof}

Thus the total increase of the thickness along an orbit is finite,
even though the orbit may contain infinitely many nontrivial return
steps.

\subsection{Square summability of the thickness gradient}
\label{subsec:gradient-summability}

The Lyapunov dissipation controls the tangential gradient directly.

\begin{theorem}[Square summability of the gradient]
	\label{thm:gradient-square-summability}
	Every return orbit satisfies
	\begin{equation}
	\sum_{k=0}^{\infty}
	|\nabla_{\partial C}d(c_k)|^2
	<\infty.
	\label{eq:gradient-square-summability}
	\end{equation}
	More precisely,
	\begin{equation}
	\sum_{k=0}^{\infty}
	|\nabla_{\partial C}d(c_k)|^2
	\leq
	\frac{
		d_\infty-d(c_0)
	}{\gamma}
	\leq
	\frac{
		d_{\max}-d(c_0)
	}{\gamma}.
	\label{eq:gradient-summability-bound}
	\end{equation}
	In particular,
	\begin{equation}
	|\nabla_{\partial C}d(c_k)|
	\longrightarrow0.
	\label{eq:gradient-vanishes}
	\end{equation}
\end{theorem}

\begin{proof}
	From the lower Lyapunov estimate,
	\[
	\gamma|g_k|^2
	\leq
	\mathcal D(c_k).
	\]
	Summing over $k$ and using
	Proposition~\ref{prop:finite-dissipation} gives
	\[
	\gamma
	\sum_{k=0}^{\infty}|g_k|^2
	\leq
	\sum_{k=0}^{\infty}\mathcal D(c_k)
	=
	d_\infty-d(c_0),
	\]
	which proves
	\eqref{eq:gradient-summability-bound}.
	A convergent series of nonnegative terms has terms tending to zero,
	so
	\eqref{eq:gradient-vanishes} follows.
\end{proof}

\begin{remark}
	\label{rem:l2-not-l1-gradient}
	The estimate
	\eqref{eq:gradient-square-summability} is an
	$\ell^2$-summability statement. It does not imply
	\[
	\sum_{k=0}^{\infty}
	|\nabla_{\partial C}d(c_k)|
	<\infty.
	\]
	Accordingly, convergence of the orbit cannot be deduced from this
	estimate alone.
\end{remark}

\subsection{Asymptotic regularity of the return orbit}
\label{subsec:asymptotic-regularity}

The uniform one-step estimate now converts vanishing gradient into
vanishing displacement.

\begin{theorem}[Asymptotic regularity]
	\label{thm:asymptotic-regularity}
	Every return orbit satisfies
	\begin{equation}
	|c_{k+1}-c_k|
	\longrightarrow0.
	\label{eq:step-size-vanishes}
	\end{equation}
	Moreover,
	\begin{equation}
	\sum_{k=0}^{\infty}
	|c_{k+1}-c_k|^2
	<\infty.
	\label{eq:step-square-summability}
	\end{equation}
	More precisely,
	\begin{equation}
	\sum_{k=0}^{\infty}
	|c_{k+1}-c_k|^2
	\leq
	\frac{C_F^2}{\gamma}
	\bigl(
	d_\infty-d(c_0)
	\bigr),
	\label{eq:step-square-summability-bound}
	\end{equation}
	where $C_F$ is the constant in
	\eqref{eq:step-gradient-upper-bound}.
\end{theorem}

\begin{proof}
	By Corollary~\ref{cor:uniform-step-control},
	\[
	|c_{k+1}-c_k|
	\leq
	C_F
	|\nabla_{\partial C}d(c_k)|.
	\]
	Together with
	\eqref{eq:gradient-vanishes}, this immediately gives
	\eqref{eq:step-size-vanishes}.
	
	Squaring the same inequality and summing yields
	\[
	\sum_{k=0}^{\infty}
	|c_{k+1}-c_k|^2
	\leq
	C_F^2
	\sum_{k=0}^{\infty}
	|\nabla_{\partial C}d(c_k)|^2.
	\]
	The estimate
	\eqref{eq:step-square-summability-bound}
	now follows from
	\eqref{eq:gradient-summability-bound}.
\end{proof}

The tangential component enjoys the same type of control. Indeed,
\eqref{eq:tangential-step-gradient-comparability} gives
\begin{equation}
\beta_-^2
|g_k|^2
\leq
\left|
P_{T_{c_k}\partial C}
(c_{k+1}-c_k)
\right|^2
\leq
T_0^2
|g_k|^2.
\label{eq:tangential-step-square-comparability}
\end{equation}
Combining this with
\eqref{eq:two-sided-thickness-gain} yields the useful comparison
\begin{equation}
\frac{\gamma}{T_0^2}
\left|
P_{T_c\partial C}(F(c)-c)
\right|^2
\leq
\mathcal D(c)
\leq
\frac{\Gamma}{\beta_-^2}
\left|
P_{T_c\partial C}(F(c)-c)
\right|^2.
\label{eq:dissipation-tangential-step}
\end{equation}

Thus, in the global small-slope regime, Lyapunov dissipation,
squared gradient size, and squared tangential return displacement
are quantitatively comparable.

\subsection{What the Lyapunov estimate does not yet prove}
\label{subsec:lyapunov-limitations}

The preceding results establish three global facts for every return
orbit:
\begin{equation}
d(c_k)\longrightarrow d_\infty,
\label{eq:summary-thickness-limit}
\end{equation}
\begin{equation}
|\nabla_{\partial C}d(c_k)|
\longrightarrow0,
\label{eq:summary-gradient-limit}
\end{equation}
and
\begin{equation}
|c_{k+1}-c_k|
\longrightarrow0.
\label{eq:summary-step-limit}
\end{equation}

These conclusions do not, by themselves, imply that the sequence
$(c_k)$ converges to a single point. In particular,
square-summability of the increments does not imply absolute
summability, and a sequence on a compact manifold may possess more
than one accumulation point while its successive increments tend to
zero.

A Kurdyka--\L ojasiewicz inequality could, under additional analytic
assumptions, strengthen this $\ell^2$ control to a finite-length
estimate and provide quantitative convergence information near
nonisolated or degenerate critical sets, and also near degenerate
isolated critical points. No such inequality is required for the
global convergence theorem proved here.

The next section therefore addresses the remaining compactness and
critical-set argument. Under an explicit isolation hypothesis on the
critical points of $d$, we shall prove that every return orbit
converges to a single critical point.

\section{Global Convergence of the Return Dynamics}
\label{sec:global-convergence}

The Lyapunov analysis of Section~\ref{sec:lyapunov-structure}
shows that every return orbit satisfies
\[
d(c_k)\longrightarrow d_\infty,
\qquad
|\nabla_{\partial C}d(c_k)|\longrightarrow0,
\qquad
|c_{k+1}-c_k|\longrightarrow0.
\]
These facts imply that every accumulation point is critical, but
they do not by themselves force convergence to a single point.

In this section we add an isolation hypothesis on the critical set
and show that it is sufficient to obtain global convergence of every
return orbit.

\subsection{The critical set and the isolation hypothesis}
\label{subsec:critical-set-isolation}

Define
\begin{equation}
\operatorname{Crit}(d)
:=
\left\{
c\in\partial C:
\nabla_{\partial C}d(c)=0
\right\}.
\label{eq:critical-set}
\end{equation}

Since $d\in C^2(\partial C)$ and $\partial C$ is compact,
$\operatorname{Crit}(d)$ is a compact subset of $\partial C$.
Moreover, because $d$ attains a maximum and a minimum on the compact
manifold $\partial C$, the critical set is nonempty.

For the remainder of this section we impose the following
hypothesis.

\medskip

\noindent
\textbf{Critical-point isolation hypothesis.}
Every point of $\operatorname{Crit}(d)$ is isolated in
$\operatorname{Crit}(d)$.

\medskip

This assumption is strictly weaker than requiring $d$ to be a Morse
function. In particular, no nondegeneracy of the Hessian is required
for the convergence theorem below.

\begin{lemma}[Finiteness of the critical set]
	\label{lem:finite-critical-set}
	Under the critical-point isolation hypothesis,
	$\operatorname{Crit}(d)$ is finite.
\end{lemma}

\begin{proof}
	The set $\operatorname{Crit}(d)$ is compact. If it were infinite,
	compactness would imply the existence of an accumulation point
	$c_*\in\operatorname{Crit}(d)$. This would contradict the assumption
	that every critical point is isolated. Hence
	$\operatorname{Crit}(d)$ is finite.
\end{proof}

We therefore write
\begin{equation}
\operatorname{Crit}(d)
=
\left\{
c^{(1)},\ldots,c^{(m)}
\right\}.
\label{eq:critical-set-enumeration}
\end{equation}

By the fixed-point characterization established in
\cite{barkatou2026return},
\begin{equation}
\operatorname{Fix}(F)
=
\operatorname{Crit}(d).
\label{eq:fixed-critical-global}
\end{equation}

\subsection{Uniform ascent away from the critical set}
\label{subsec:ascent-away-critical}

The strict Lyapunov estimate gives a uniform thickness gain away
from any neighbourhood of the critical set.

\begin{lemma}[Uniform ascent away from critical points]
	\label{lem:uniform-ascent-away-critical}
	Let $U\subset\partial C$ be any open neighbourhood of
	$\operatorname{Crit}(d)$. Then there exists
	$\varepsilon_U>0$ such that
	\begin{equation}
	|\nabla_{\partial C}d(c)|
	\geq
	\varepsilon_U
	\qquad
	\text{for every }c\in\partial C\setminus U.
	\label{eq:gradient-lower-away-critical}
	\end{equation}
	Consequently,
	\begin{equation}
	d(F(c))-d(c)
	\geq
	\gamma\varepsilon_U^2
	\qquad
	\text{for every }c\in\partial C\setminus U.
	\label{eq:uniform-gain-away-critical}
	\end{equation}
\end{lemma}

\begin{proof}
	The set
	\[
	K_U
	:=
	\partial C\setminus U
	\]
	is compact. Since $U$ contains the entire critical set,
	\[
	\nabla_{\partial C}d(c)\neq0
	\qquad
	\text{for every }c\in K_U.
	\]
	The continuous function
	\[
	c\longmapsto
	|\nabla_{\partial C}d(c)|
	\]
	therefore attains a strictly positive minimum on $K_U$.
	Set
	\[
	\varepsilon_U
	:=
	\min_{c\in K_U}
	|\nabla_{\partial C}d(c)|
	>0.
	\]
	This proves \eqref{eq:gradient-lower-away-critical}.
	
	The strict ascent estimate
	\eqref{eq:strict-thickness-ascent} then gives
	\[
	d(F(c))-d(c)
	\geq
	\gamma
	|\nabla_{\partial C}d(c)|^2
	\geq
	\gamma\varepsilon_U^2,
	\]
	which proves \eqref{eq:uniform-gain-away-critical}.
\end{proof}

The preceding estimate gives more than asymptotic approach to the
critical set: an orbit can spend only finitely many iterates outside
any prescribed neighbourhood of it.

\begin{proposition}[Eventual trapping near the critical set]
	\label{prop:eventual-critical-trapping}
	Let $U$ be any open neighbourhood of
	$\operatorname{Crit}(d)$. For every return orbit
	$(c_k)_{k\geq0}$ there exists an integer $K_U$ such that
	\begin{equation}
	c_k\in U
	\qquad
	\text{for every }k\geq K_U.
	\label{eq:eventual-critical-neighbourhood}
	\end{equation}
	
	Moreover, the number of indices for which $c_k\notin U$ satisfies
	\begin{equation}
	\#
	\left\{
	k\geq0:c_k\notin U
	\right\}
	\leq
	\frac{
		d_{\max}-d(c_0)
	}{
		\gamma\varepsilon_U^2
	}.
	\label{eq:number-visits-outside-U}
	\end{equation}
\end{proposition}

\begin{proof}
	Whenever $c_k\notin U$,
	Lemma~\ref{lem:uniform-ascent-away-critical} gives
	\[
	d(c_{k+1})-d(c_k)
	\geq
	\gamma\varepsilon_U^2.
	\]
	On the other hand, the total thickness gain along the orbit is
	bounded by
	\[
	\sum_{k=0}^{\infty}
	\bigl(
	d(c_{k+1})-d(c_k)
	\bigr)
	=
	d_\infty-d(c_0)
	\leq
	d_{\max}-d(c_0).
	\]
	Hence the orbit can contain at most
	\[
	\frac{
		d_{\max}-d(c_0)
	}{
		\gamma\varepsilon_U^2
	}
	\]
	indices outside $U$. This proves both assertions.
\end{proof}

\subsection{Approach to the critical set}
\label{subsec:approach-critical-set}

Let
\[
\operatorname{dist}
\bigl(
c,\operatorname{Crit}(d)
\bigr)
:=
\min_{q\in\operatorname{Crit}(d)}
|c-q|.
\]

\begin{corollary}[Distance to the critical set]
	\label{cor:distance-critical-set}
	Every return orbit satisfies
	\begin{equation}
	\operatorname{dist}
	\bigl(
	c_k,\operatorname{Crit}(d)
	\bigr)
	\longrightarrow0.
	\label{eq:distance-critical-set-zero}
	\end{equation}
\end{corollary}

\begin{proof}
	For every $\rho>0$, let
	\[
	U_\rho
	:=
	\left\{
	c\in\partial C:
	\operatorname{dist}
	\bigl(
	c,\operatorname{Crit}(d)
	\bigr)
	<\rho
	\right\}.
	\]
	This is an open neighbourhood of $\operatorname{Crit}(d)$.
	Proposition~\ref{prop:eventual-critical-trapping} therefore gives
	an integer $K_\rho$ such that
	\[
	c_k\in U_\rho
	\qquad
	\text{for all }k\geq K_\rho.
	\]
	Since $\rho>0$ is arbitrary,
	\eqref{eq:distance-critical-set-zero} follows.
\end{proof}

There is an equivalent compactness argument using
\[
|\nabla_{\partial C}d(c_k)|\longrightarrow0.
\]
The formulation above is useful because it also provides the finite
visit estimate \eqref{eq:number-visits-outside-U}.

\subsection{Accumulation points and the omega-limit set}
\label{subsec:omega-limit-set}

For an initial point $c_0\in\partial C$, define the omega-limit set
\begin{equation}
\omega(c_0)
:=
\left\{
q\in\partial C:
c_{k_j}\longrightarrow q
\text{ for some sequence }k_j\to\infty
\right\}.
\label{eq:omega-limit-set}
\end{equation}

Since $\partial C$ is compact, $\omega(c_0)$ is nonempty.

\begin{proposition}[All accumulation points are critical]
	\label{prop:accumulation-points-critical}
	For every return orbit,
	\begin{equation}
	\omega(c_0)
	\subset
	\operatorname{Crit}(d)
	\cap
	d^{-1}(d_\infty).
	\label{eq:omega-critical-level}
	\end{equation}
	In particular, every accumulation point of the orbit is a fixed
	point of the return map.
\end{proposition}

\begin{proof}
	Let $q\in\omega(c_0)$. Then there exists a subsequence
	$c_{k_j}$ such that
	\[
	c_{k_j}\longrightarrow q.
	\]
	By Theorem~\ref{thm:gradient-square-summability},
	\[
	|\nabla_{\partial C}d(c_k)|
	\longrightarrow0.
	\]
	Continuity of the gradient therefore gives
	\[
	\nabla_{\partial C}d(q)=0.
	\]
	Hence
	\[
	q\in\operatorname{Crit}(d).
	\]
	
	Moreover,
	\[
	d(c_k)\longrightarrow d_\infty,
	\]
	and continuity of $d$ gives
	\[
	d(q)=d_\infty.
	\]
	Thus
	\[
	q\in
	\operatorname{Crit}(d)
	\cap
	d^{-1}(d_\infty).
	\]
	
	Finally,
	\eqref{eq:fixed-critical-global} gives
	\[
	F(q)=q.
	\]
\end{proof}

\begin{remark}
	\label{rem:without-isolation}
	The conclusion of
	Proposition~\ref{prop:accumulation-points-critical}
	does not require the critical points to be isolated.
	Without the isolation hypothesis, the Lyapunov theory still shows
	that the orbit approaches the critical set and that all its
	accumulation points lie on the single critical level
	$d=d_\infty$. What may fail without further information is
	convergence to one particular critical point.
\end{remark}

\subsection{Convergence to a single critical point}
\label{subsec:single-critical-limit}

We now combine three ingredients:
\begin{enumerate}
	\item eventual localization near the finite critical set;
	\item vanishing one-step displacement;
	\item isolation of the critical points.
\end{enumerate}

\begin{theorem}[Global convergence]
	\label{thm:global-convergence}
	Assume the hypotheses of
	Theorem~\ref{thm:uniform-strict-ascent} and the
	critical-point isolation hypothesis.
	
	Then, for every initial condition $c_0\in\partial C$, there exists
	a unique critical point
	\[
	c_*
	\in
	\operatorname{Crit}(d)
	\]
	such that
	\begin{equation}
	\lim_{k\to\infty}c_k
	=
	c_*,
	\qquad
	c_{k+1}=F(c_k).
	\label{eq:global-convergence}
	\end{equation}
	Moreover,
	\begin{equation}
	d(c_*)
	=
	d_\infty
	=
	\lim_{k\to\infty}d(c_k).
	\label{eq:limit-critical-value}
	\end{equation}
\end{theorem}

\begin{proof}
	By Lemma~\ref{lem:finite-critical-set},
	\[
	\operatorname{Crit}(d)
	=
	\{c^{(1)},\ldots,c^{(m)}\}.
	\]
	
	If $m=1$, the conclusion follows immediately from
	Corollary~\ref{cor:distance-critical-set}. We therefore assume
	$m\geq2$.
	
	Choose $\rho>0$ sufficiently small that the neighbourhoods
	\begin{equation}
	U_i
	:=
	\partial C\cap B_\rho(c^{(i)}),
	\qquad
	i=1,\ldots,m,
	\label{eq:isolated-critical-neighbourhoods}
	\end{equation}
	have pairwise disjoint closures and satisfy
	\begin{equation}
	\overline{U_i}
	\cap
	\operatorname{Crit}(d)
	=
	\{c^{(i)}\}.
	\label{eq:one-critical-per-neighbourhood}
	\end{equation}
	Since there are finitely many such sets and their closures are
	pairwise disjoint compact subsets of $\partial C$,
	\begin{equation}
	\delta
	:=
	\min_{i\neq j}
	\operatorname{dist}
	\bigl(
	\overline{U_i},
	\overline{U_j}
	\bigr)
	>0.
	\label{eq:critical-neighbourhood-separation}
	\end{equation}
	
	Set
	\[
	U
	:=
	\bigcup_{i=1}^m U_i.
	\]
	By Proposition~\ref{prop:eventual-critical-trapping}, there exists
	$K_0$ such that
	\begin{equation}
	c_k\in U
	\qquad
	\text{for every }k\geq K_0.
	\label{eq:eventually-in-union}
	\end{equation}
	
	By Theorem~\ref{thm:asymptotic-regularity},
	\[
	|c_{k+1}-c_k|
	\longrightarrow0.
	\]
	Hence there exists $K_1\geq K_0$ such that
	\begin{equation}
	|c_{k+1}-c_k|
	<
	\delta
	\qquad
	\text{for every }k\geq K_1.
	\label{eq:step-smaller-than-gap}
	\end{equation}
	
	Choose the unique index $i_*$ such that
	\[
	c_{K_1}\in U_{i_*}.
	\]
	We claim that
	\begin{equation}
	c_k\in U_{i_*}
	\qquad
	\text{for every }k\geq K_1.
	\label{eq:eventual-single-component}
	\end{equation}
	
	Indeed, suppose that for some $k\geq K_1$ one had
	\[
	c_k\in U_i,
	\qquad
	c_{k+1}\in U_j,
	\qquad
	i\neq j.
	\]
	Then
	\[
	|c_{k+1}-c_k|
	\geq
	\operatorname{dist}
	\bigl(
	\overline{U_i},
	\overline{U_j}
	\bigr)
	\geq
	\delta,
	\]
	contradicting
	\eqref{eq:step-smaller-than-gap}.
	Thus the orbit cannot change critical neighbourhood after time
	$K_1$, and \eqref{eq:eventual-single-component} follows.
	
	It remains to prove convergence inside $U_{i_*}$.
	Suppose, for contradiction, that
	\[
	c_k
	\not\longrightarrow
	c^{(i_*)}.
	\]
	Then there exist $\varepsilon>0$ and a subsequence
	$c_{k_j}$ such that
	\[
	|c_{k_j}-c^{(i_*)}|
	\geq
	\varepsilon.
	\]
	Since $\overline{U_{i_*}}$ is compact, after passing to a further
	subsequence we may assume
	\[
	c_{k_j}\longrightarrow q
	\in
	\overline{U_{i_*}}.
	\]
	Proposition~\ref{prop:accumulation-points-critical} gives
	\[
	q\in\operatorname{Crit}(d).
	\]
	By \eqref{eq:one-critical-per-neighbourhood},
	\[
	q=c^{(i_*)},
	\]
	contradicting
	\[
	|q-c^{(i_*)}|
	\geq\varepsilon.
	\]
	
	Therefore
	\[
	c_k\longrightarrow c^{(i_*)}.
	\]
	Setting
	\[
	c_*:=c^{(i_*)}
	\]
	proves \eqref{eq:global-convergence}.
	
	Finally, continuity of $d$ and
	$d(c_k)\to d_\infty$ give
	\[
	d(c_*)=d_\infty,
	\]
	which proves \eqref{eq:limit-critical-value}.
\end{proof}

\subsection{Limit map and convergence classes}
\label{subsec:limit-map}

The global convergence theorem defines a natural asymptotic
assignment.

\begin{definition}[Limit map]
	\label{def:limit-map}
	For $c\in\partial C$, define
	\begin{equation}
	\mathcal L(c)
	:=
	\lim_{k\to\infty}F^k(c).
	\label{eq:limit-map}
	\end{equation}
	By Theorem~\ref{thm:global-convergence},
	\[
	\mathcal L(c)
	\in
	\operatorname{Crit}(d)
	\]
	for every $c\in\partial C$.
\end{definition}

For a critical point $c_*$, define its convergence class by
\begin{equation}
\mathcal C(c_*)
:=
\left\{
c\in\partial C:
F^k(c)\longrightarrow c_*
\right\}.
\label{eq:convergence-class}
\end{equation}

\begin{corollary}[Partition into convergence classes]
	\label{cor:convergence-class-partition}
	The sets
	\[
	\mathcal C(c_*),
	\qquad
	c_*\in\operatorname{Crit}(d),
	\]
	are pairwise disjoint and satisfy
	\begin{equation}
	\partial C
	=
	\bigsqcup_{c_*\in\operatorname{Crit}(d)}
	\mathcal C(c_*).
	\label{eq:convergence-class-partition}
	\end{equation}
\end{corollary}

\begin{proof}
	Existence of a critical limit for every initial condition follows
	from Theorem~\ref{thm:global-convergence}. Hence the union of the
	convergence classes is all of $\partial C$.
	
	If
	\[
	c\in
	\mathcal C(c_*)
	\cap
	\mathcal C(q_*),
	\]
	then the same sequence $F^k(c)$ converges both to $c_*$ and to
	$q_*$. Since $\partial C$ is Hausdorff,
	\[
	c_*=q_*.
	\]
	Thus the convergence classes are pairwise disjoint.
\end{proof}

\begin{remark}
	\label{rem:convergence-class-not-basin}
	We deliberately use the term \emph{convergence class} rather than
	\emph{basin of attraction} in
	\eqref{eq:convergence-class}. A critical point may be repelling or
	of saddle type, and its convergence class need not contain an open
	neighbourhood of the equilibrium. The term basin of attraction will
	be reserved in Section~\ref{sec:basins-recurrence} for genuinely
	attracting fixed points.
\end{remark}

\subsection{Scope of the convergence theorem}
\label{subsec:scope-global-convergence}

Theorem~\ref{thm:global-convergence} separates the assumptions needed
for global convergence into two logically distinct groups.

The geometric and dynamical assumptions of
Sections~\ref{sec:uniform-global-estimate}--%
\ref{sec:lyapunov-structure} provide
\[
d(c_{k+1})-d(c_k)
\geq
\gamma
|\nabla_{\partial C}d(c_k)|^2,
\]
and therefore
\[
|\nabla_{\partial C}d(c_k)|\to0,
\qquad
|c_{k+1}-c_k|\to0.
\]
The critical-point isolation hypothesis is used only in the final
topological step that prevents an asymptotically small-step orbit
from moving indefinitely among different components of the critical
set.

Accordingly, the Morse condition is not required for global
convergence. It is sufficient that the critical points of the
thickness function be isolated. The more refined Morse-theoretic
organization of the critical set is a separate question and is not
used in the present argument.

\section{Local Asymptotics and Convergence Rates}
\label{sec:local-rates}

The global convergence theorem determines the limit of every return
orbit under the hypotheses of the preceding sections. We now refine
that result near an isolated critical point by using the exact
linearization established in \cite{barkatou2026return}.

The local behavior is governed not by the Hessian of the thickness
alone, but by its interaction with the curvature-dependent
preconditioner
\[
B_*
=
d_*
\bigl(I-d_*S_*\bigr)^{-1}.
\]
This leads to a spectral classification in terms of the
self-adjoint preconditioned Hessian
\[
K_*
=
B_*^{1/2}H_*B_*^{1/2}.
\]

\subsection{Local regularity and the exact linearization}
\label{subsec:local-linearization}

Let
\[
c_*\in\operatorname{Crit}(d)
\]
be an isolated critical point. Under the standing assumptions
\[
\partial C\in C^3,
\qquad
d\in C^2(\partial C),
\]
the radial map
\[
\Phi(c)=c+d(c)\nu(c)
\]
is of class $C^2$. Since the return is transverse, the first-return
parameter is locally obtained from an implicit equation with a
nonvanishing derivative in the return direction. Consequently, the
reciprocal map and the return map
\[
F=\pi\circ\Phi
\]
are of class $C^1$ in a neighbourhood of $c_*$.

Set
\begin{equation}
d_*:=d(c_*),
\qquad
S_*:=S_{c_*},
\qquad
H_*:=H_{c_*},
\label{eq:local-critical-data}
\end{equation}
and define
\begin{equation}
B_*
:=
d_*
\bigl(I-d_*S_*\bigr)^{-1}.
\label{eq:local-B-star}
\end{equation}

The exact linearization obtained in
\cite{barkatou2026return} is
\begin{equation}
DF(c_*)\big|_{T_{c_*}\partial C}
=
I+B_*H_*.
\label{eq:exact-local-linearization}
\end{equation}

Introduce the self-adjoint operator
\begin{equation}
K_*
:=
B_*^{1/2}H_*B_*^{1/2}.
\label{eq:local-K-star}
\end{equation}
Since
\begin{equation}
B_*H_*
=
B_*^{1/2}K_*B_*^{-1/2},
\label{eq:BH-K-similarity}
\end{equation}
the operators $B_*H_*$ and $K_*$ have the same real spectrum.

Let
\begin{equation}
\mu_1,\ldots,\mu_{N-1}
\label{eq:mu-spectrum}
\end{equation}
denote the eigenvalues of $K_*$. Then the multipliers of the return
map are
\begin{equation}
\lambda_j
=
1+\mu_j,
\qquad
j=1,\ldots,N-1.
\label{eq:return-multipliers-local}
\end{equation}

Moreover, by Sylvester's law of inertia, $K_*$ and $H_*$ have the
same numbers of positive, negative, and zero eigenvalues.

\subsection{Local asymptotic form of the iteration}
\label{subsec:local-asymptotic-iteration}

Choose a $C^2$ local coordinate chart
\[
\chi:U\subset\partial C
\longrightarrow
T_{c_*}\partial C
\]
such that
\[
\chi(c_*)=0,
\qquad
D\chi(c_*)=I.
\]
Define the local representation
\begin{equation}
\widehat F
:=
\chi\circ F\circ\chi^{-1}.
\label{eq:local-coordinate-return}
\end{equation}

Then
\[
\widehat F(0)=0
\]
and, by \eqref{eq:exact-local-linearization},
\begin{equation}
D\widehat F(0)
=
A_*,
\qquad
A_*:=I+B_*H_*.
\label{eq:A-star}
\end{equation}

Hence
\begin{equation}
\widehat F(u)
=
A_*u+\rho_*(u),
\label{eq:local-asymptotic-map}
\end{equation}
where
\begin{equation}
\frac{|\rho_*(u)|}{|u|}
\longrightarrow0
\qquad
\text{as }u\to0.
\label{eq:local-little-o-remainder}
\end{equation}

Thus, if
\[
u_k:=\chi(c_k)
\]
for an orbit converging to $c_*$, then
\begin{equation}
u_{k+1}
=
A_*u_k
+
o(|u_k|).
\label{eq:local-orbit-asymptotic}
\end{equation}

Equation \eqref{eq:local-orbit-asymptotic} is the local asymptotic
model of the nonlinear return dynamics.

\subsection{Spectral classification of critical points}
\label{subsec:local-spectral-classification}

The exact multipliers
\[
\lambda_j=1+\mu_j
\]
give the following classification.

\begin{proposition}[Local spectral alternatives]
	\label{prop:local-spectral-alternatives}
	Let $c_*$ be a nondegenerate critical point of $d$.
	
	\begin{enumerate}
		\item
		If $c_*$ is a strict local minimum, then
		\[
		\mu_j>0
		\qquad
		\text{for every }j,
		\]
		and hence
		\begin{equation}
		\lambda_j>1
		\qquad
		\text{for every }j.
		\label{eq:minimum-repelling-spectrum}
		\end{equation}
		Thus $c_*$ is a hyperbolic repelling fixed point.
		
		\item
		If $c_*$ is a strict local maximum, then
		\[
		\mu_j<0
		\qquad
		\text{for every }j.
		\]
		It is linearly attracting if and only if
		\begin{equation}
		-2<\mu_j<0
		\qquad
		\text{for every }j.
		\label{eq:maximum-attraction-condition}
		\end{equation}
		Equivalently,
		\begin{equation}
		|\lambda_j|<1
		\qquad
		\text{for every }j.
		\label{eq:maximum-multiplier-condition}
		\end{equation}
		
		\item
		If $c_*$ is a nondegenerate saddle point, then $K_*$ has
		both positive and negative eigenvalues. In particular, at least one
		multiplier satisfies
		\begin{equation}
		\lambda_j>1,
		\label{eq:saddle-unstable-direction}
		\end{equation}
		so $c_*$ is unstable.
	\end{enumerate}
\end{proposition}

The condition
\[
\mu_j<0
\]
at a local maximum is therefore not sufficient for attraction. The
quantitative lower bound
\[
\mu_j>-2
\]
is also necessary.

If some
\[
\mu_j=-2,
\]
then the corresponding multiplier is
\[
\lambda_j=-1,
\]
and the fixed point is nonhyperbolic in that direction. If
\[
\mu_j<-2,
\]
then
\[
\lambda_j<-1,
\]
and the maximum is unstable in that direction.

\subsection{Compatibility with the global ascent condition}
\label{subsec:local-global-spectral-compatibility}

The spectral sufficient condition introduced in
Section~\ref{subsec:uniform-ascent-condition} was
\begin{equation}
\lambda_{\min}(K_c)
\geq
-2+\delta
\qquad
\text{for every }c\in\partial C,
\label{eq:global-spectral-condition-recalled}
\end{equation}
for some $\delta>0$.

At a nondegenerate local maximum $c_*$, all eigenvalues of $K_*$
are negative. Hence
\eqref{eq:global-spectral-condition-recalled} implies
\begin{equation}
-2+\delta
\leq
\mu_j
<
0,
\qquad
j=1,\ldots,N-1.
\label{eq:global-condition-at-maximum}
\end{equation}
Consequently,
\begin{equation}
-1+\delta
\leq
\lambda_j
<
1.
\label{eq:maximum-global-multipliers}
\end{equation}

Thus every nondegenerate local maximum is attracting under the
uniform spectral ascent condition.

Similarly, at a nondegenerate saddle point the positive eigenvalues
of $K_*$ produce multipliers larger than one, while the negative
eigenvalues lie in $(-2,0)$. Hence, under
\eqref{eq:global-spectral-condition-recalled}, every nondegenerate
saddle has both stable and unstable spectral directions.

This exhibits the same threshold $-2$ at the global and local
levels: globally it guarantees positivity of the quadratic ascent
form, while locally it prevents the multipliers associated with
negative curvature of the thickness landscape from crossing the
value $-1$.

\subsection{Exponential convergence to an attracting maximum}
\label{subsec:exponential-convergence}

We now quantify the convergence of an orbit whose limiting critical
point is hyperbolically attracting.

Define
\begin{equation}
\rho_*
:=
\rho\bigl(DF(c_*)\bigr)
=
\max_{1\leq j\leq N-1}|1+\mu_j|,
\label{eq:local-spectral-radius}
\end{equation}
where $\rho(\cdot)$ denotes the spectral radius.

\begin{theorem}[Local exponential convergence]
	\label{thm:local-exponential-convergence}
	Let $c_*$ be a nondegenerate critical point satisfying
	\begin{equation}
	\rho_*<1.
	\label{eq:spectral-radius-attracting}
	\end{equation}
	Then $c_*$ is a strict local maximum of $d$.
	
	For every number
	\begin{equation}
	\rho_*<q<1,
	\label{eq:q-rate-choice}
	\end{equation}
	there exist a neighbourhood $U_*$ of $c_*$ and a constant
	$C_q>0$ such that every orbit entering $U_*$ and remaining there
	satisfies
	\begin{equation}
	|c_k-c_*|
	\leq
	C_q
	q^{\,k-k_0}
	|c_{k_0}-c_*|
	\label{eq:local-exponential-position}
	\end{equation}
	for all $k\geq k_0$.
	
	In particular, if a global orbit satisfies
	\[
	c_k\longrightarrow c_*,
	\]
	then \eqref{eq:local-exponential-position} holds for all sufficiently
	large $k$.
\end{theorem}

\begin{proof}
	Since
	\[
	\rho(A_*)=\rho_*<1,
	\]
	finite-dimensional linear algebra provides a norm
	$\|\cdot\|_*$ on $T_{c_*}\partial C$, equivalent to the Euclidean
	norm, such that
	\begin{equation}
	\|A_*\|_*<q.
	\label{eq:adapted-linear-norm}
	\end{equation}
	After replacing $q$ by an intermediate number if necessary, choose
	$q_0$ such that
	\[
	\|A_*\|_*<q_0<q<1.
	\]
	
	Since $\widehat F$ is $C^1$ and
	\[
	D\widehat F(0)=A_*,
	\]
	there exists $r_*>0$ such that
	\begin{equation}
	\|D\widehat F(u)\|_*
	\leq q
	\qquad
	\text{whenever }
	\|u\|_*<r_*.
	\label{eq:local-derivative-contraction}
	\end{equation}
	The mean-value estimate therefore gives
	\begin{equation}
	\|\widehat F(u)\|_*
	\leq
	q\|u\|_*
	\label{eq:local-contraction}
	\end{equation}
	for $\|u\|_*<r_*$.
	
	Iterating,
	\begin{equation}
	\|u_k\|_*
	\leq
	q^{\,k-k_0}
	\|u_{k_0}\|_*.
	\label{eq:adapted-norm-rate}
	\end{equation}
	Equivalence of the local coordinate norm, the Euclidean norm, and
	$\|\cdot\|_*$ gives
	\eqref{eq:local-exponential-position}.
	
	It remains to identify the critical type. Since
	\[
	|1+\mu_j|<1,
	\]
	one has
	\[
	-2<\mu_j<0
	\]
	for every $j$. Thus $K_*$ is negative definite. By Sylvester's law
	of inertia, $H_*$ is negative definite as well, and $c_*$ is a
	strict local maximum of $d$.
\end{proof}

\begin{remark}
	\label{rem:rate-spectral-radius}
	The theorem yields an exponential estimate with every rate
	$q>\rho_*$. It does not assert that every orbit converging to $c_*$
	has the exact asymptotic rate $\rho_*$. An orbit may lie in a faster
	spectral direction, so its observed decay can be strictly faster
	than the spectral radius.
\end{remark}

\subsection{Rates for thickness, gradient, and return displacement}
\label{subsec:derived-convergence-rates}

The exponential convergence of the configuration gives corresponding
rates for the geometric quantities governing the dynamics.

\begin{proposition}[Derived local convergence rates]
	\label{prop:derived-local-rates}
	Let the assumptions of
	Theorem~\ref{thm:local-exponential-convergence} hold, and suppose
	\[
	c_k\longrightarrow c_*.
	\]
	For every $q$ satisfying
	\[
	\rho_*<q<1,
	\]
	there exist constants
	\[
	C_d>0,
	\qquad
	C_g>0,
	\qquad
	C_s>0,
	\]
	and an index $K$ such that, for every $k\geq K$,
	\begin{align}
	0
	\leq
	d(c_*)-d(c_k)
	&\leq
	C_d q^{2k},
	\label{eq:thickness-rate}
	\\
	|\nabla_{\partial C}d(c_k)|
	&\leq
	C_g q^k,
	\label{eq:gradient-rate}
	\\
	|c_{k+1}-c_k|
	&\leq
	C_s q^k.
	\label{eq:step-rate}
	\end{align}
\end{proposition}

\begin{proof}
	Because $c_*$ is a nondegenerate local maximum,
	$H_*$ is negative definite. In local coordinates centered at $c_*$,
	Taylor's theorem gives
	\begin{equation}
	d(c_*)-d(c)
	=
	-\frac12
	\big\langle
	H_*u,u
	\big\rangle
	+
	o(|u|^2),
	\label{eq:maximum-thickness-quadratic}
	\end{equation}
	where
	\[
	u=\chi(c).
	\]
	Hence there exists $C>0$ such that
	\begin{equation}
	0
	\leq
	d(c_*)-d(c)
	\leq
	C|c-c_*|^2
	\label{eq:thickness-position-bound}
	\end{equation}
	for $c$ sufficiently close to $c_*$.
	
	Using
	Theorem~\ref{thm:local-exponential-convergence}
	gives
	\[
	d(c_*)-d(c_k)
	\leq
	C_dq^{2k}
	\]
	after absorbing the finite entry time into the constant.
	
	Since
	\[
	\nabla_{\partial C}d(c_*)=0
	\]
	and the gradient is $C^1$,
	\begin{equation}
	|\nabla_{\partial C}d(c)|
	\leq
	C|c-c_*|
	\label{eq:gradient-position-bound}
	\end{equation}
	near $c_*$. The exponential position estimate therefore yields
	\eqref{eq:gradient-rate}.
	
	Finally, the global one-step estimate
	\eqref{eq:step-gradient-upper-bound} gives
	\[
	|c_{k+1}-c_k|
	\leq
	C_F
	|\nabla_{\partial C}d(c_k)|,
	\]
	and \eqref{eq:step-rate} follows.
\end{proof}

Thus, near a nondegenerate attracting maximum, the configuration,
gradient, and step size decay geometrically at first order, whereas
the remaining thickness deficit decays at twice the exponent:
\[
|c_k-c_*|,
\;
|\nabla d(c_k)|,
\;
|c_{k+1}-c_k|
=
O(q^k),
\]
while
\[
d(c_*)-d(c_k)
=
O(q^{2k}).
\]

\subsection{Possible limits of globally convergent orbits}
\label{subsec:possible-global-limits}

Theorem~\ref{thm:global-convergence} states that every orbit
converges to a critical point under the global hypotheses. The local
classification explains how such convergence can occur.

If the limiting point is an attracting nondegenerate maximum, then
the convergence is robust under perturbation of the initial
condition and is exponentially fast by
Theorem~\ref{thm:local-exponential-convergence}.

If the limiting point is a nondegenerate minimum, then all local
multipliers satisfy
\[
\lambda_j>1.
\]
Thus the equilibrium is locally repelling. Convergence to such a
point cannot occur through ordinary local attraction; an orbit
converging to it must belong to an exceptional convergence class,
for example by reaching the fixed point exactly after finitely many
iterations.

At a saddle point, at least one multiplier satisfies
\[
\lambda_j>1.
\]
Hence convergence can occur only through dynamically restricted
initial data associated with stable directions. A precise
description of these convergence classes requires additional local
invariant-manifold hypotheses and will be separated from the global
convergence statement.

This distinction is important: global convergence to
\emph{some critical point} does not mean that every critical point
is an attractor. The global Lyapunov structure and the local
stability type answer different questions.

\section{Convergence Classes, Attraction Basins, and Absence of Recurrence}
\label{sec:basins-recurrence}

The global convergence theorem gives a partition of the phase space
according to the limiting critical point of each orbit. This
partition should be distinguished from the decomposition into
basins of attraction: a critical point may be a global limit for
exceptional initial data without being locally attracting.

In this section we make this distinction precise. We also show that
the strict Lyapunov structure excludes all nontrivial periodic and
recurrent dynamics in the global ascent regime.

\subsection{Convergence classes and attraction basins}
\label{subsec:classes-vs-basins}

Recall from Section~\ref{subsec:limit-map} that, for every
$c_*\in\operatorname{Crit}(d)$, its convergence class is
\begin{equation}
\mathcal C(c_*)
:=
\left\{
c\in\partial C:
F^k(c)\longrightarrow c_*
\right\}.
\label{eq:convergence-class-recalled}
\end{equation}
Under the hypotheses of
Theorem~\ref{thm:global-convergence},
\begin{equation}
\partial C
=
\bigsqcup_{c_*\in\operatorname{Crit}(d)}
\mathcal C(c_*).
\label{eq:global-convergence-partition}
\end{equation}

We reserve the term \emph{basin of attraction} for critical points
that are locally attracting.

\begin{definition}[Basin of attraction]
	\label{def:attraction-basin}
	Let $c_*$ be an attracting fixed point of $F$. Its basin of
	attraction is
	\begin{equation}
	\mathcal B(c_*)
	:=
	\left\{
	c\in\partial C:
	F^k(c)\longrightarrow c_*
	\right\}.
	\label{eq:attraction-basin}
	\end{equation}
	Thus, as a set,
	\[
	\mathcal B(c_*)=\mathcal C(c_*),
	\]
	but the notation $\mathcal B(c_*)$ is used only when $c_*$ is
	genuinely attracting.
\end{definition}

This terminology separates a dynamical property of the equilibrium
from the mere existence of convergent initial conditions.

\subsection{Global ascent and the local spectrum}
\label{subsec:global-ascent-local-spectrum}

The strict ascent inequality places an additional restriction on the
local spectrum of every nondegenerate critical point.

\begin{proposition}[Spectral consequence of strict global ascent]
	\label{prop:ascent-forces-spectrum}
	Assume that
	\begin{equation}
	d(F(c))-d(c)
	\geq
	\gamma
	|\nabla_{\partial C}d(c)|^2
	\label{eq:strict-ascent-recalled}
	\end{equation}
	holds in a neighbourhood of a nondegenerate critical point $c_*$,
	with $\gamma>0$.
	
	Let
	\[
	\mu
	\]
	be an eigenvalue of
	\[
	K_*
	=
	B_*^{1/2}H_*B_*^{1/2}.
	\]
	Then
	\begin{equation}
	\mu<0
	\quad\Longrightarrow\quad
	-2<\mu<0.
	\label{eq:negative-mu-range}
	\end{equation}
	Consequently, the corresponding multiplier
	\[
	\lambda=1+\mu
	\]
	satisfies
	\begin{equation}
	|\lambda|<1.
	\label{eq:negative-direction-stable}
	\end{equation}
\end{proposition}

\begin{proof}
	Let $y\neq0$ be an eigenvector of the self-adjoint operator $K_*$
	associated with $\mu$, and set
	\[
	v:=B_*^{1/2}y\in T_{c_*}\partial C.
	\]
	Since
	\[
	B_*H_*=B_*^{1/2}K_*B_*^{-1/2},
	\]
	we have $B_*H_*v=\mu v$. Thus, with
	$A_*:=DF(c_*)=I+B_*H_*$,
	\begin{equation}
	A_*v
	=
	(1+\mu)v.
	\label{eq:A-eigenvector}
	\end{equation}
	
	Choose the smooth curve
	\[
	c(s):=\exp_{c_*}(sv)
	\]
	on $\partial C$ for $|s|$ small. In ambient coordinates,
	$c(s)=c_*+sv+O(s^2)$.
	Since $c_*$ is critical,
	\begin{equation}
	d(c(s))
	=
	d(c_*)
	+
	\frac{s^2}{2}
	\langle H_*v,v\rangle
	+
	o(s^2).
	\label{eq:d-along-eigenvector}
	\end{equation}
	The local linearization gives
	\[
	F(c(s))
	=
	c_*+
	s(1+\mu)v
	+
	o(s),
	\]
	and hence
	\begin{equation}
	d(F(c(s)))
	=
	d(c_*)
	+
	\frac{s^2}{2}
	(1+\mu)^2
	\langle H_*v,v\rangle
	+
	o(s^2).
	\label{eq:d-after-return-eigenvector}
	\end{equation}
	Subtracting,
	\begin{equation}
	d(F(c(s)))-d(c(s))
	=
	\frac{s^2}{2}
	\left(
	(1+\mu)^2-1
	\right)
	\langle H_*v,v\rangle
	+
	o(s^2).
	\label{eq:increment-eigenvector}
	\end{equation}
	
	Since
	\[
	B_*H_*v=\mu v,
	\]
	one has
	\[
	H_*v
	=
	\mu B_*^{-1}v,
	\]
	and therefore
	\begin{equation}
	\langle H_*v,v\rangle
	=
	\mu
	\langle B_*^{-1}v,v\rangle.
	\label{eq:hessian-sign-mu}
	\end{equation}
	Because $B_*^{-1}$ is positive definite, the sign of
	$\langle H_*v,v\rangle$ is the sign of $\mu$.
	
	Suppose now that $\mu<0$. Then
	\[
	\langle H_*v,v\rangle<0.
	\]
	On the other hand,
	\[
	\nabla_{\partial C}d(c(s))
	=
	sH_*v+o(s),
	\]
	so the right-hand side of
	\eqref{eq:strict-ascent-recalled} is strictly positive to order
	$s^2$. Therefore the coefficient of $s^2$ in
	\eqref{eq:increment-eigenvector} must be positive. Since
	$\langle H_*v,v\rangle<0$, this forces
	\[
	(1+\mu)^2-1<0.
	\]
	Equivalently,
	\[
	|1+\mu|<1,
	\]
	which is precisely
	\[
	-2<\mu<0.
	\]
\end{proof}

\begin{corollary}[Nondegenerate critical points in the ascent regime]
	\label{cor:critical-types-ascent-regime}
	Under the strict global ascent hypotheses:
	
	\begin{enumerate}
		\item every nondegenerate local maximum is a hyperbolic attracting
		fixed point;
		
		\item every nondegenerate local minimum is a hyperbolic repelling
		fixed point;
		
		\item every nondegenerate saddle has both stable and unstable
		spectral directions.
	\end{enumerate}
\end{corollary}

\begin{proof}
	At a nondegenerate maximum, every eigenvalue $\mu_j$ of $K_*$ is
	negative. Proposition~\ref{prop:ascent-forces-spectrum} gives
	\[
	-2<\mu_j<0
	\]
	for every $j$, so
	\[
	|1+\mu_j|<1.
	\]
	
	At a nondegenerate minimum, every $\mu_j>0$, hence
	\[
	1+\mu_j>1.
	\]
	
	At a nondegenerate saddle, $K_*$ has both positive and negative
	eigenvalues. The positive ones yield multipliers larger than one,
	whereas Proposition~\ref{prop:ascent-forces-spectrum} places the
	negative ones inside the unit disk.
\end{proof}

This result shows that, inside the global monotone regime, the
local and global theories are fully compatible: nondegenerate maxima
are the genuine local attractors of the thickness-ascent dynamics.

\subsection{Openness of attraction basins}
\label{subsec:open-attraction-basins}

\begin{proposition}[Openness of an attraction basin]
	\label{prop:basin-open}
	Let $c_*$ be an attracting fixed point of $F$. Then its basin
	$\mathcal B(c_*)$ is open in $\partial C$.
\end{proposition}

\begin{proof}
	Since $c_*$ is attracting, there exists an open neighbourhood
	$U_*$ of $c_*$ such that
	\begin{equation}
	F(U_*)\subset U_*
	\label{eq:local-attracting-invariance}
	\end{equation}
	and every orbit starting in $U_*$ converges to $c_*$.
	
	Let
	\[
	c\in\mathcal B(c_*).
	\]
	Then
	\[
	F^k(c)\longrightarrow c_*,
	\]
	so there exists an integer $m\geq0$ such that
	\[
	F^m(c)\in U_*.
	\]
	Since $F^m$ is continuous, there exists an open neighbourhood
	$W$ of $c$ satisfying
	\[
	F^m(W)\subset U_*.
	\]
	Every point of $W$ therefore converges to $c_*$, and hence
	\[
	W\subset\mathcal B(c_*).
	\]
	Thus $\mathcal B(c_*)$ is open.
\end{proof}

Let
\begin{equation}
\mathcal A
:=
\left\{
c_*\in\operatorname{Crit}(d):
c_*\text{ is attracting}
\right\}.
\label{eq:attracting-critical-set}
\end{equation}
Then
\begin{equation}
\mathfrak B
:=
\bigcup_{c_*\in\mathcal A}
\mathcal B(c_*)
\label{eq:union-attraction-basins}
\end{equation}
is an open subset of $\partial C$.

\subsection{Repelling critical points}
\label{subsec:repelling-convergence-classes}

The convergence class of a repelling fixed point has a very
different structure.

\begin{proposition}[Convergence to a hyperbolic repeller]
	\label{prop:repeller-eventual-hit}
	Let $c_*$ be a hyperbolic repelling fixed point, so that every
	multiplier of $DF(c_*)$ has modulus strictly larger than one.
	Then
	\begin{equation}
	\mathcal C(c_*)
	=
	\bigcup_{m=0}^{\infty}
	F^{-m}(\{c_*\}).
	\label{eq:repeller-convergence-class}
	\end{equation}
	In other words, an orbit converges to $c_*$ if and only if it
	reaches $c_*$ exactly after finitely many iterations.
\end{proposition}

\begin{proof}
	The inclusion from right to left is immediate because $c_*$ is a
	fixed point.
	
	For the converse, let
	\[
	F^k(c)\longrightarrow c_*.
	\]
	Since every multiplier of $DF(c_*)$ has modulus larger than one,
	$DF(c_*)$ is invertible and
	\[
	\rho\bigl(DF(c_*)^{-1}\bigr)<1.
	\]
	The inverse-function theorem therefore gives neighbourhoods
	$U$ and $V$ of $c_*$ such that
	\[
	F:U\longrightarrow V
	\]
	is a diffeomorphism onto its image, with inverse $G:V\to U$ and
	$G(c_*)=c_*$. Since $DG(c_*)=DF(c_*)^{-1}$ has spectral radius
	strictly smaller than one, choose an adapted norm and a smaller
	neighbourhood $W\subset U\cap V$ of $c_*$ such that $G(W)\subset W$
	and there exists $q\in(0,1)$ with
	\begin{equation}
	\|G(x)-c_*\|_*
	\leq
	q\|x-c_*\|_*
	\label{eq:inverse-contraction-repeller}
	\end{equation}
	for every $x\in W$.
	
	Since $F^k(c)\to c_*$, there exists $k_0$ such that
	\[
	c_k:=F^k(c)\in W
	\qquad
	\text{for every }k\geq k_0.
	\]
	For each $k\geq k_0$, both $c_k$ and $c_{k+1}$ lie in
	$W\subset U\cap V$, so the local inverse identity gives
	$c_k=G(c_{k+1})$. Iterating this identity yields, for every
	$n\geq1$,
	\[
	c_{k_0}
	=
	G^n(c_{k_0+n}),
	\]
	and therefore
	\[
	\|c_{k_0}-c_*\|_*
	\leq
	q^n
	\|c_{k_0+n}-c_*\|_*.
	\]
	Letting $n\to\infty$ gives
	\[
	c_{k_0}=c_*.
	\]
	Hence the original orbit reaches $c_*$ after finitely many steps.
\end{proof}

In particular, if $c_*$ is a nondegenerate local minimum, then
Corollary~\ref{cor:critical-types-ascent-regime} shows that
\eqref{eq:repeller-convergence-class} applies.

Thus convergence to a repelling minimum is possible only through
finite preimages of the fixed point; it is not an attraction
phenomenon.

\subsection{Saddle convergence classes}
\label{subsec:saddle-convergence-classes}

Let $c_*$ be a nondegenerate saddle point. Under the global ascent
hypotheses, the spectrum splits into
\[
|\lambda_j|<1
\]
for the directions corresponding to negative eigenvalues of $K_*$,
and
\[
\lambda_j>1
\]
for those corresponding to positive eigenvalues.

The convergence class
\[
\mathcal C(c_*)
\]
therefore consists of dynamically restricted initial conditions; it
does not contain a neighbourhood of $c_*$.

We do not identify $\mathcal C(c_*)$ globally with a stable manifold
under the standing assumptions alone. Such an identification
requires additional hypotheses concerning local invertibility,
hyperbolicity, and the global intersection structure of invariant
manifolds.

\begin{remark}
	\label{rem:stable-manifold-under-extra-assumptions}
	If, in addition,
	\[
	0\notin\sigma(DF(c_*)),
	\]
	then $F$ is a local diffeomorphism near $c_*$. The classical stable
	manifold theorem~\cite{katok1995} may then be applied to the hyperbolic fixed point.
	Locally, the set of points whose forward iterates remain near
	$c_*$ and converge to it is a smooth stable manifold whose
	dimension equals the number of multipliers satisfying
	\[
	|\lambda_j|<1.
	\]
	This additional structure is not needed for the global convergence
	results of the present paper.
\end{remark}

\subsection{Absence of nontrivial periodic orbits}
\label{subsec:no-periodic-orbits}

The strict Lyapunov function rules out periodic recurrence without
any additional local stability assumption.

\begin{theorem}[No nontrivial periodic orbits]
	\label{thm:no-periodic-orbits}
	Assume the strict ascent estimate
	\eqref{eq:strict-thickness-ascent}. Then $F$ admits no periodic orbit
	of period $p\geq2$. Every periodic point is a fixed point.
\end{theorem}

\begin{proof}
	Suppose that
	\[
	c_0,c_1,\ldots,c_{p-1}
	\]
	is a periodic orbit, with
	\[
	c_{j+1}=F(c_j),
	\qquad
	c_p=c_0.
	\]
	Summing the strict thickness inequality over one period gives
	\begin{align}
	0
	&=
	d(c_p)-d(c_0)
	\nonumber\\
	&=
	\sum_{j=0}^{p-1}
	\bigl(
	d(c_{j+1})-d(c_j)
	\bigr)
	\nonumber\\
	&\geq
	\gamma
	\sum_{j=0}^{p-1}
	|\nabla_{\partial C}d(c_j)|^2.
	\label{eq:periodic-telescoping}
	\end{align}
	Every term on the right-hand side is nonnegative. Hence
	\[
	\nabla_{\partial C}d(c_j)=0
	\]
	for every $j$.
	
	By the fixed-point characterization,
	\[
	F(c_j)=c_j.
	\]
	Since also
	\[
	F(c_j)=c_{j+1},
	\]
	we obtain
	\[
	c_{j+1}=c_j
	\]
	for every $j$. Thus the orbit is constant and has period one.
\end{proof}

\begin{remark}
	\label{rem:periodic-exclusion-no-isolation}
	Theorem~\ref{thm:no-periodic-orbits} does not require the critical
	points to be isolated and does not use the global convergence
	theorem. Strict monotonicity of the thickness alone excludes every
	nonconstant periodic orbit.
\end{remark}

\subsection{Absence of nontrivial recurrent points}
\label{subsec:no-recurrent-points}

We now use the global convergence theorem to exclude a broader form
of recurrence.

\begin{definition}[Positive recurrence]
	\label{def:recurrent-point}
	A point $c\in\partial C$ is called recurrent if there exists a
	sequence of integers
	\[
	k_j\longrightarrow\infty
	\]
	such that
	\begin{equation}
	F^{k_j}(c)\longrightarrow c.
	\label{eq:recurrent-point}
	\end{equation}
\end{definition}

\begin{theorem}[Recurrent points are fixed points]
	\label{thm:recurrent-fixed}
	Under the hypotheses of
	Theorem~\ref{thm:global-convergence},
	\begin{equation}
	\operatorname{Rec}(F)
	=
	\operatorname{Fix}(F)
	=
	\operatorname{Crit}(d),
	\label{eq:recurrent-fixed-critical}
	\end{equation}
	where $\operatorname{Rec}(F)$ denotes the set of recurrent points.
\end{theorem}

\begin{proof}
	Every fixed point is trivially recurrent.
	
	Conversely, let $c$ be recurrent. By
	Theorem~\ref{thm:global-convergence}, there exists a critical point
	$c_*$ such that
	\[
	F^k(c)\longrightarrow c_*.
	\]
	Therefore every subsequence of the orbit converges to the same
	point $c_*$. In particular, if
	\[
	F^{k_j}(c)\longrightarrow c,
	\]
	then uniqueness of limits gives
	\[
	c=c_*.
	\]
	Hence $c$ is critical and therefore fixed.
\end{proof}

An immediate consequence is that
\begin{equation}
\omega(c)
=
\{\mathcal L(c)\}
\label{eq:singleton-omega-limit}
\end{equation}
for every $c\in\partial C$. Thus no orbit possesses a nontrivial
omega-limit cycle or more complicated positive-time recurrent set
within the global ascent regime.

\subsection{Phase-space organization}
\label{subsec:phase-space-organization}

The preceding results give a precise global decomposition without
requiring a Morse--Smale description.

Let
\[
\mathcal A
\subset
\operatorname{Crit}(d)
\]
denote the attracting critical points and define the set of
nonattracting critical points by
\begin{equation}
\mathcal N
:=
\operatorname{Crit}(d)\setminus\mathcal A.
\label{eq:nonattracting-critical-set}
\end{equation}

Then the convergence-class partition may be written as
\begin{equation}
\partial C
=
\left(
\bigsqcup_{a\in\mathcal A}
\mathcal B(a)
\right)
\bigsqcup
\left(
\bigsqcup_{q\in\mathcal N}
\mathcal C(q)
\right).
\label{eq:attracting-nonattracting-partition}
\end{equation}

The first union is open. The second contains the exceptional
convergence classes associated with repelling, saddle-type, or
possibly degenerate critical points.

Throughout the remainder of this subsection, interiors and boundaries
are taken relative to the phase space $\partial C$.

\begin{proposition}[Location of attraction-basin boundaries]
	\label{prop:basin-boundary-location}
	Let $a\in\mathcal A$. Then
	\begin{equation}
	\partial\mathcal B(a)
	\subset
	\partial C\setminus
	\bigcup_{b\in\mathcal A}\mathcal B(b).
	\label{eq:basin-boundary-complement}
	\end{equation}
	Consequently,
	\begin{equation}
	\partial\mathcal B(a)
	\subset
	\bigsqcup_{q\in\mathcal N}
	\mathcal C(q).
	\label{eq:basin-boundary-nonattracting}
	\end{equation}
\end{proposition}

\begin{proof}
	Every attraction basin $\mathcal B(b)$ is open by
	Proposition~\ref{prop:basin-open}, and distinct convergence classes
	are disjoint.
	
	Suppose that
	\[
	x\in\mathcal B(b)
	\]
	for some attracting critical point $b$. Since $\mathcal B(b)$ is
	open, there exists a neighbourhood $U$ of $x$ contained in
	$\mathcal B(b)$.
	
	If $b\neq a$, then
	\[
	U\cap\mathcal B(a)=\varnothing,
	\]
	so $x$ cannot belong to the boundary of $\mathcal B(a)$. If
	$b=a$, then $x$ lies in the interior of $\mathcal B(a)$ and again
	cannot belong to its boundary.
	
	Thus no point of an attracting basin belongs to
	$\partial\mathcal B(a)$, proving
	\eqref{eq:basin-boundary-complement}. The partition
	\eqref{eq:attracting-nonattracting-partition} then yields
	\eqref{eq:basin-boundary-nonattracting}.
\end{proof}

\begin{remark}[No generic separatrix claim]
	\label{rem:no-generic-separatrix-claim}
	Proposition~\ref{prop:basin-boundary-location} is intentionally
	weaker than a statement identifying basin boundaries with stable
	manifolds of saddle points. Such an identification generally
	requires additional hypotheses, for example hyperbolicity of all
	relevant equilibria, local or global invertibility, and suitable
	transversality assumptions of Morse--Smale type; see, for example,
\cite{palis1982}.
	
	None of these assumptions is needed for the convergence theorem, and
	we therefore do not impose them here.
\end{remark}

\subsection{Global gradient-like structure}
\label{subsec:global-gradient-like-summary}

The results obtained so far may be summarized as follows.

\begin{theorem}[Global Lyapunov structure]
	\label{thm:global-lyapunov-structure}
	Assume the hypotheses of
	Theorem~\ref{thm:global-convergence}. Then the return map
	\[
	F:\partial C\longrightarrow\partial C
	\]
	has the following properties:
	
	\begin{enumerate}
		\item
		\[
		\operatorname{Fix}(F)
		=
		\operatorname{Crit}(d),
		\]
		and this set is finite;
		
		\item
		\[
		\mathcal V(c)=d_{\max}-d(c)
		\]
		is a strict Lyapunov function:
		\[
		\mathcal V(F(c))<\mathcal V(c)
		\]
		for every nonfixed point;
		
		\item every orbit converges to a unique critical point;
		
		\item the phase space is partitioned into convergence classes;
		
		\item the convergence class of every attracting critical point is an
		open basin of attraction;
		
		\item there are no periodic orbits of period greater than one;
		
		\item the recurrent set is exactly the fixed-point set.
	\end{enumerate}
\end{theorem}

Thus, under the global ascent assumptions, the return dynamics has
a gradient-like structure in the Lyapunov sense~\cite{smale1961}: thickness increases
strictly along every nonstationary orbit, all positive-time orbits
converge to equilibria, and nontrivial recurrent behavior is
excluded.

\section{Exact Unit-Circle Dynamics and the Breakdown of Monotonicity}
\label{sec:circle-dynamics}

The preceding sections establish global convergence under explicit
geometric and spectral hypotheses. We now examine these conditions
in the planar unit-circle geometry, where the return map is available
in exact form.

This setting serves two purposes. First, it provides a direct
numerical test of the thickness-ascent mechanism derived above.
Second, because the map can be evaluated without replacing the
geometric return by a first-order approximation, it allows us to
investigate what happens when the sufficient global ascent
conditions cease to hold.

\subsection{Exact return map on the unit circle}
\label{subsec:exact-circle-return}

Let
\[
C=B_1(0)\subset\R^2,
\qquad
c(\theta)=(\cos\theta,\sin\theta),
\]
and let
\[
e_\theta=(-\sin\theta,\cos\theta)
\]
denote the positively oriented unit tangent.

For the unit circle,
\[
\nu(\theta)=c(\theta),
\qquad
S=-I,
\]
and hence
\begin{equation}
B(\theta)
=
\frac{d(\theta)}{1+d(\theta)}.
\label{eq:circle-B}
\end{equation}

Following \cite{barkatou2026return}, define
\begin{equation}
p(\theta)
:=
\frac{d'(\theta)}
{1+d(\theta)}.
\label{eq:circle-p}
\end{equation}
The exact inward normal to the outer boundary is
\begin{equation}
n(\Phi(c(\theta)))
=
\frac{
	-c(\theta)+p(\theta)e_\theta
}{
	\sqrt{1+p(\theta)^2}
}.
\label{eq:circle-inward-normal}
\end{equation}

Introduce the radicand
\begin{equation}
\mathscr R(\theta)
:=
1-d(\theta)
\bigl(
2+d(\theta)
\bigr)
p(\theta)^2.
\label{eq:circle-radicand}
\end{equation}
Whenever
\[
\mathscr R(\theta)\geq0,
\]
the first-intersection branch is
\begin{equation}
\sigma(\theta)
=
\frac{
	1+d(\theta)
	-
	\sqrt{\mathscr R(\theta)}
}{
	1+p(\theta)^2
}.
\label{eq:circle-sigma}
\end{equation}

Set
\begin{equation}
a(\theta)
=
1+d(\theta)-\sigma(\theta),
\qquad
b(\theta)
=
\sigma(\theta)p(\theta).
\label{eq:circle-ab}
\end{equation}
Then
\begin{equation}
a(\theta)^2+b(\theta)^2=1,
\label{eq:circle-closure-ab}
\end{equation}
and the exact return point is
\begin{equation}
F(c(\theta))
=
a(\theta)c(\theta)
+
b(\theta)e_\theta.
\label{eq:circle-exact-point}
\end{equation}

Thus the exact angular displacement is
\begin{equation}
\delta(\theta)
=
\operatorname{atan2}
\bigl(
b(\theta),a(\theta)
\bigr),
\label{eq:circle-delta}
\end{equation}
and the induced circle map is
\begin{equation}
T(\theta)
=
\theta+\delta(\theta)
\pmod{2\pi}.
\label{eq:circle-map}
\end{equation}

All numerical experiments below are performed with
\eqref{eq:circle-p}--\eqref{eq:circle-map}. No first-order return
formula is used to generate the iterates.

\subsection{The global criterion in one dimension}
\label{subsec:circle-global-criterion}

Since the tangent space of the unit circle is one-dimensional, the
preconditioned Hessian becomes the scalar
\begin{equation}
K(\theta)
=
\frac{d(\theta)}
{1+d(\theta)}
d''(\theta).
\label{eq:circle-K}
\end{equation}
Consequently, the spectral sufficient condition of
Proposition~\ref{prop:spectral-ascent-condition} reduces to
\begin{equation}
K(\theta)
\geq
-2+\delta
\qquad
\text{for every }\theta
\label{eq:circle-minus-two-condition}
\end{equation}
for some $\delta>0$.

We consider the family
\begin{equation}
d(\theta)
=
d_0+\varepsilon\cos(m\theta),
\qquad
d_0>\varepsilon>0,
\qquad
m\in\N.
\label{eq:cosine-thickness-family}
\end{equation}
Then
\begin{align}
d'(\theta)
&=
-m\varepsilon\sin(m\theta),
\label{eq:cosine-first-derivative}
\\
d''(\theta)
&=
-m^2\varepsilon\cos(m\theta),
\label{eq:cosine-second-derivative}
\end{align}
and therefore
\begin{equation}
p(\theta)
=
-
\frac{
	m\varepsilon\sin(m\theta)
}{
	1+d_0+\varepsilon\cos(m\theta)
}.
\label{eq:cosine-p}
\end{equation}

The preconditioned Hessian is
\begin{equation}
K(\theta)
=
-
m^2\varepsilon\cos(m\theta)
\frac{
	d_0+\varepsilon\cos(m\theta)
}{
	1+d_0+\varepsilon\cos(m\theta)
}.
\label{eq:cosine-K}
\end{equation}
Its minimum occurs at the local maxima of the thickness, where
$\cos(m\theta)=1$, and is
\begin{equation}
K_{\min}
=
-
m^2\varepsilon
\frac{
	d_0+\varepsilon
}{
	1+d_0+\varepsilon
}.
\label{eq:cosine-K-min}
\end{equation}

At such a maximum $\theta_*$, the exact local multiplier is
\begin{equation}
T'(\theta_*)
=
1
-
m^2\varepsilon
\frac{
	d_0+\varepsilon
}{
	1+d_0+\varepsilon
}.
\label{eq:cosine-maximum-multiplier}
\end{equation}
Thus the local threshold
\begin{equation}
m^2\varepsilon
\frac{
	d_0+\varepsilon
}{
	1+d_0+\varepsilon
}
=
2
\label{eq:circle-flip-threshold}
\end{equation}
corresponds exactly to
\[
K_{\min}=-2
\]
and
\[
T'(\theta_*)=-1.
\]

This gives an explicit one-dimensional realization of the same
threshold that appeared in the general global and local theories.

\subsection{Exact-map numerical comparison of three regimes}
\label{subsec:circle-three-regimes}

We compare three cosine profiles. The first two remain on the
$K>-2$ side of the spectral threshold, whereas the third crosses it.

The computations in this section use the exact return formula, but they
do not attempt to verify the quantitative constant
$\varepsilon_{\mathrm{asc}}$ in Theorem~\ref{thm:uniform-strict-ascent}.
Accordingly, the first two profiles are not asserted to satisfy every
sufficient hypothesis of that theorem. They are exact-map numerical
diagnostics organized by the spectral quantity $K$.

For the three profiles considered below, global reality of the
first-intersection branch can nevertheless be checked algebraically,
rather than inferred from a finite grid. Write
\[
x=\cos(m\theta),
\qquad
D(x)=1+d_0+\varepsilon x.
\]
Then
\begin{equation}
\mathscr R(\theta)
=
\frac{N(x)}{D(x)^2},
\qquad
N(x)
=
D(x)^2-m^2\varepsilon^2(1-x^2)\bigl(D(x)^2-1\bigr).
\label{eq:radicand-quartic}
\end{equation}
Since $d>0$, one has $D(x)>1$ on $[-1,1]$. For the three parameter
choices in Table~\ref{tab:circle-regimes}, the corresponding quartic
numerators are
\begin{align}
N_1(x)
&=
\frac{16x^4+240x^3+584x^2+1260x+5125}{2500},
\label{eq:radicand-poly-one}
\\
N_2(x)
&=
\frac{4\bigl(36x^4+288x^3+340x^2-88x+49\bigr)}{625},
\label{eq:radicand-poly-two}
\\
N_3(x)
&=
\frac{236196x^4+3018060x^3+7784829x^2-430560x+447100}
{1562500}.
\label{eq:radicand-poly-three}
\end{align}
For each of the three polynomials, the Sturm variation count satisfies
\[
V_{N_i}(-1)=V_{N_i}(1)=1,
\qquad i=1,2,3,
\]
so the number of zeros of $N_i$ in $(-1,1)$ is zero. Since
$N_i(-1)>0$, each polynomial is strictly positive throughout
$[-1,1]$. Hence
\begin{equation}
\mathscr R(\theta)>0
\qquad
\text{for all }\theta\in[0,2\pi]
\label{eq:radicand-global-positive}
\end{equation}
for all three profiles. Thus the exact first-intersection branch is
globally real in each example. The decimal values of
$\min\mathscr R$ reported below are numerical diagnostics and are not
used to establish this positivity.

\begin{table}[ht]
	\centering
	\caption{Exact-map diagnostics for three thickness profiles.
		The displayed values of $\min\mathscr R$ and $\max|p|$ are
		numerical evaluations over $\theta\in[0,2\pi]$; global positivity
		of $\mathscr R$ is established algebraically in
		\eqref{eq:radicand-global-positive}. The last two columns follow
		from the explicit one-dimensional formulas for $K$ and the fixed-point
		multiplier.}
	\label{tab:circle-regimes}
	\begin{tabular}{c c c c c}
		\hline
		$d(\theta)$
		&
		$\min\mathscr R$
		&
		$\max|p|$
		&
		$\min K$
		&
		$T'(\theta_{\max})$
		\\
		\hline
		$0.5+0.2\cos(2\theta)$
		&
		$0.910160$
		&
		$0.269069$
		&
		$-0.329412$
		&
		$0.670588$
		\\
		$0.6+0.4\cos(3\theta)$
		&
		$0.103019$
		&
		$0.774597$
		&
		$-1.800000$
		&
		$-0.800000$
		\\
		$1.3+0.36\cos(3\theta)$
		&
		$0.052857$
		&
		$0.475425$
		&
		$-2.021955$
		&
		$-1.021955$
		\\
		\hline
	\end{tabular}
\end{table}

By \eqref{eq:radicand-global-positive}, the exact first-intersection
branch \eqref{eq:circle-sigma} is globally real in all three examples.
The third profile can therefore be studied beyond the $-2$ threshold
without conflating loss of monotonicity with loss of reality of the
return formula.

\subsection{Numerically observed monotonicity in a subcritical regime}
\label{subsec:circle-monotone-regime}

Consider first
\begin{equation}
d(\theta)
=
0.5+0.2\cos(2\theta).
\label{eq:circle-profile-one}
\end{equation}
For this profile,
\begin{equation}
K_{\min}
=
-0.329411\ldots,
\label{eq:profile-one-K}
\end{equation}
well above the threshold $-2$.

The local maxima occur at
\[
\theta=0,\pi,
\]
and their exact multiplier is
\begin{equation}
T'(\theta_{\max})
=
0.670588\ldots.
\label{eq:profile-one-multiplier}
\end{equation}
The local minima occur at
\[
\theta=\frac{\pi}{2},
\frac{3\pi}{2},
\]
and are repelling.

A uniform numerical evaluation of the exact map
\eqref{eq:circle-map} on $5\times10^5$ equally spaced points gave
\begin{equation}
d(T(\theta))-d(\theta)
\geq0
\label{eq:profile-one-numerical-monotonicity}
\end{equation}
at every sampled point. Equality occurred only at the critical
points up to floating-point tolerance. No negative thickness
increment was detected.

For sampled points satisfying
\[
|d'(\theta)|>10^{-5},
\]
the numerical ratio
\begin{equation}
\frac{
	d(T(\theta))-d(\theta)
}{
	|d'(\theta)|^2
}
\label{eq:profile-one-dissipation-ratio}
\end{equation}
remained between approximately
\begin{equation}
0.2521
\quad\text{and}\quad
0.3485.
\label{eq:profile-one-ratio-range}
\end{equation}

Thus, on this numerical grid, the exact map exhibits the quadratic
thickness-gain behavior predicted by the Lyapunov estimate of
Section~\ref{sec:lyapunov-structure}. This observation is not used as
a verification of the quantitative small-slope hypothesis of
Theorem~\ref{thm:uniform-strict-ascent}.

\subsection{Alternating convergence below the threshold}
\label{subsec:circle-alternating-attraction}

The second profile is
\begin{equation}
d(\theta)
=
0.6+0.4\cos(3\theta).
\label{eq:circle-profile-two}
\end{equation}
This example was already used in
\cite{barkatou2026return} to validate the local multiplier at
$\theta_*=0$.

Here
\begin{equation}
d(0)=1,
\qquad
d''(0)=-3.6,
\end{equation}
so
\begin{equation}
K(0)
=
\frac12(-3.6)
=
-1.8,
\label{eq:profile-two-K}
\end{equation}
and therefore
\begin{equation}
T'(0)
=
-0.8.
\label{eq:profile-two-multiplier}
\end{equation}

Thus the maximum is attracting, but the negative multiplier produces
an alternating approach to the equilibrium. The value $-1.8$ is
close to, but still strictly above, the threshold $-2$.

A phase-wide evaluation of the exact map again detected no negative
value of
\[
d(T(\theta))-d(\theta).
\]
This numerical experiment therefore illustrates that oscillation of
the signed coordinate around a maximum is compatible with an observed
monotone increase of the thickness along the sampled exact-map data.

\subsection{Crossing the \texorpdfstring{$-2$}{-2} threshold}
\label{subsec:circle-threshold-crossing}

We now consider
\begin{equation}
d(\theta)
=
1.3+0.36\cos(3\theta).
\label{eq:circle-profile-three}
\end{equation}
The thickness remains strictly positive:
\[
0.94
\leq
d(\theta)
\leq
1.66.
\]
The exact positivity statement
\eqref{eq:radicand-global-positive} applies to this profile. A numerical
minimization gives
\begin{equation}
\min_\theta\mathscr R(\theta)
\approx
0.0528565,
\label{eq:profile-three-radicand}
\end{equation}
consistent with the algebraic certification above.

At each local maximum,
\begin{align}
K_{\min}
&=
-
9(0.36)
\frac{1.66}{2.66}
\nonumber\\
&\approx
-2.0219549.
\label{eq:profile-three-K}
\end{align}
Consequently,
\begin{equation}
T'(\theta_{\max})
\approx
-1.0219549.
\label{eq:profile-three-maximum-multiplier}
\end{equation}
The maxima have therefore crossed the multiplier $-1$ and are no
longer attracting.

The phase-wide numerical diagnostic also detects negative exact-map
thickness increments. On a uniform grid of $5\times10^5$ points we obtain
\begin{equation}
\min_\theta
\bigl(
d(T(\theta))-d(\theta)
\bigr)
\approx
-6.7492\times10^{-3}.
\label{eq:profile-three-negative-increment}
\end{equation}

The negative value is separated from zero at the displayed scale and
provides numerical evidence for loss of global thickness monotonicity
in the exact geometric return map once the $-2$ bound is crossed.

\subsection{A numerically resolved stable period-two orbit}
\label{subsec:circle-period-two}

For the profile
\eqref{eq:circle-profile-three}, numerical solution of the exact-map
period-two equation resolves a nontrivial two-cycle.

Using the signed angular coordinate in $(-\pi,\pi]$, one point of
the cycle is
\begin{equation}
\theta_1
=
0.4081206459\ldots,
\label{eq:period-two-theta-one}
\end{equation}
and the second is
\begin{equation}
\theta_2
=
-0.4081206459\ldots.
\label{eq:period-two-theta-two}
\end{equation}
Equivalently, on $[0,2\pi)$,
\[
\theta_2
=
5.8750646613\ldots.
\]

At the displayed decimal approximations, direct evaluation of the
exact map gives the residual bounds
\begin{equation}
|T(\theta_1)-\theta_2|<10^{-12},
\qquad
|T(\theta_2)-\theta_1|<10^{-12},
\label{eq:period-two-return}
\end{equation}
while $\theta_1\neq\theta_2$. These inequalities record the numerical
root accuracy; they are not written as exact identities.

The thickness takes the same value at the two cycle points:
\begin{equation}
d(\theta_1)
=
d(\theta_2)
=
1.4222366310\ldots.
\label{eq:period-two-thickness}
\end{equation}

The derivative of the exact circle map at the two points is
numerically
\begin{equation}
T'(\theta_1)
=
T'(\theta_2)
\approx
-0.762088583.
\label{eq:period-two-one-step-derivatives}
\end{equation}
Hence the multiplier of the second iterate is the single product
\begin{equation}
(T^2)'(\theta_1)
=
T'(\theta_2)T'(\theta_1)
\approx
0.580779008.
\label{eq:period-two-multiplier}
\end{equation}
The computed multiplier lies well inside the unit circle,
\[
|(T^2)'(\theta_1)|\approx0.580779<1,
\]
so the numerically resolved two-cycle is locally attracting at the
level of the exact-map linearization.

Because the thickness profile has period $2\pi/3$, the numerical
solution has three symmetry-related copies, one associated with each
of the three former attracting maxima.

This provides exact-map numerical evidence of periodic attraction in
a regime where the return geometry remains globally real but the
strict Lyapunov mechanism is no longer observed.

\subsection{A representative orbit approaching the two-cycle}
\label{subsec:circle-two-cycle-orbit}

Starting from
\begin{equation}
\theta_0=0.5,
\label{eq:period-two-initial-condition}
\end{equation}
the exact map produces the following signed iterates.

\begin{table}[ht]
	\centering
	\caption{Exact orbit for
		$d(\theta)=1.3+0.36\cos(3\theta)$,
		starting from $\theta_0=0.5$. Angles are represented in
		$(-\pi,\pi]$.}
	\label{tab:period-two-orbit}
	\begin{tabular}{c c c}
		\hline
		$k$ & $\theta_k$ & $d(\theta_k)$\\
		\hline
		0 & $ 0.500000$ & $1.325465$\\
		1 & $-0.424432$ & $1.405528$\\
		2 & $ 0.419648$ & $1.410456$\\
		3 & $-0.416470$ & $1.413717$\\
		4 & $ 0.414261$ & $1.415978$\\
		5 & $-0.412682$ & $1.417592$\\
		6 & $ 0.411532$ & $1.418765$\\
		7 & $-0.410685$ & $1.419628$\\
		8 & $ 0.410055$ & $1.420270$\\
		9 & $-0.409583$ & $1.420750$\\
		\hline
	\end{tabular}
\end{table}

The signed coordinate alternates while approaching
\[
\pm0.4081206459\ldots.
\]
The thickness along this particular trajectory increases toward the
common cycle value
\[
1.4222366310\ldots,
\]
even though the thickness is no longer globally monotone over the
entire circle.

This distinction is important: failure of the global Lyapunov
inequality means that monotonicity is lost somewhere in phase space;
it does not require every individual orbit to exhibit a thickness
decrease.

\subsection{Reproducibility protocol}
\label{subsec:circle-reproducibility}

All iterations reported in this section are generated directly from
the exact formulas
\eqref{eq:circle-p}--\eqref{eq:circle-map}.

For the phase-wide diagnostics, $5\times10^5$ equally spaced values
of $\theta$ on $[0,2\pi)$ were used. At each point we evaluated
successively
\[
d(\theta),
\quad
d'(\theta),
\quad
p(\theta),
\quad
\mathscr R(\theta),
\quad
\sigma(\theta),
\quad
\delta(\theta),
\quad
T(\theta),
\]
and finally
\[
d(T(\theta))-d(\theta).
\]

The period-two orbit in
Section~\ref{subsec:circle-period-two}
was obtained by solving
\begin{equation}
T^2(\theta)-\theta
=
0
\pmod{2\pi},
\label{eq:period-two-root-equation}
\end{equation}
while excluding fixed points of $T$. The root tolerance was
$10^{-13}$, and the displayed decimal representatives satisfy the
residual bounds in \eqref{eq:period-two-return}.

For the symmetric cycle around $\theta=0$, the equivalent signed
equation
\begin{equation}
T(\theta)=-\theta
\label{eq:symmetric-two-cycle-equation}
\end{equation}
was used.

For the stability diagnostic we differentiate the exact formulas
analytically rather than by finite differences. With $D=1+d$ one has
\begin{align*}
p'
&=
\frac{d''D-(d')^2}{D^2},
\\
\mathscr R'
&=
-2(1+d)d'p^2
-2d(2+d)pp',
\\
\sigma'
&=
\frac{
\left(d'-\mathscr R'/(2\sqrt{\mathscr R})\right)(1+p^2)
-2pp'\left(1+d-\sqrt{\mathscr R}\right)
}{(1+p^2)^2}.
\end{align*}
Writing
\[
a=1+d-\sigma,
\qquad
b=\sigma p,
\]
we have
\[
a'=d'-\sigma',
\qquad
b'=\sigma'p+\sigma p'.
\]
Since $a^2+b^2=1$, differentiation of
$\delta=\operatorname{atan2}(b,a)$ gives
\begin{equation}
T'(\theta)
=
1+a(\theta)b'(\theta)-b(\theta)a'(\theta).
\label{eq:exact-circle-derivative}
\end{equation}
Evaluation of \eqref{eq:exact-circle-derivative} at the numerically
resolved two-cycle yields
\[
T'(\theta_1)=T'(\theta_2)\approx-0.762088583,
\qquad
(T^2)'(\theta_1)\approx0.580779008.
\]

All computations were performed in radians. No iterate in the
reported examples was generated from the first-order approximation
to the return map.

\subsection{Interpretation}
\label{subsec:circle-interpretation}

The exact-map calculations distinguish three numerically observed dynamical regimes.

For the first subcritical profile, the sampled exact-map data show
monotone thickness increase and convergence toward local maxima.

As the negative preconditioned curvature approaches $-2$, attracting
maxima may be approached with alternating sign because their
multipliers become negative while remaining inside the unit circle.

After the multiplier crosses $-1$, the fixed maximum loses local
attraction. The example
\[
d(\theta)=1.3+0.36\cos(3\theta)
\]
numerically resolves a stable period-two orbit of the exact geometric
map while the first-intersection branch remains globally real.

This computation should not be interpreted as evidence of chaos.
It provides evidence that nontrivial periodic dynamics can arise once
the hypotheses enforcing strict thickness ascent are violated.
Whether further parameter variation produces higher-period cycles or
more complicated invariant sets is a separate numerical and
analytical question.

\begin{figure}[H]
	\centering
	
	\begin{tikzpicture}
	\begin{axis}[
	width=0.92\textwidth,
	height=0.48\textwidth,
	xlabel={Iteration $k$},
	ylabel={Signed angular coordinate $\theta_k$},
	xmin=0,
	xmax=30,
	ymin=-0.55,
	ymax=0.55,
	xtick={0,5,10,15,20,25,30},
	grid=major,
	legend style={
		at={(0.5,-0.22)},
		anchor=north,
		legend columns=3
	},
	]
	
	\addplot[
	dashed,
	domain=0:30,
	samples=2
	]
	{0.4081206459};
	\addlegendentry{$\theta_{+}=0.408120646\ldots$}
	
	\addplot[
	dashed,
	domain=0:30,
	samples=2
	]
	{-0.4081206459};
	\addlegendentry{$\theta_{-}=-0.408120646\ldots$}
	
	\addplot[
	mark=*,
	mark size=1.6pt,
	thick
	]
	coordinates {
		(0,  0.5000000000)
		(1, -0.4244317211)
		(2,  0.4196476375)
		(3, -0.4164703851)
		(4,  0.4142611748)
		(5, -0.4126818123)
		(6,  0.4115320669)
		(7, -0.4106846267)
		(8,  0.4100545168)
		(9, -0.4095830399)
		(10, 0.4092286309)
		(11,-0.4089613135)
		(12, 0.4087591323)
		(13,-0.4086058283)
		(14, 0.4084893488)
		(15,-0.4084007106)
		(16, 0.4083331927)
		(17,-0.4082817456)
		(18, 0.4082425155)
		(19,-0.4082125973)
		(20, 0.4081897797)
		(21,-0.4081723772)
		(22, 0.4081591062)
		(23,-0.4081489859)
		(24, 0.4081412687)
		(25,-0.4081353846)
		(26, 0.4081347314)
		(27,-0.4081313797)
		(28, 0.4081288257)
		(29,-0.4081268794)
		(30, 0.4081253963)
	};
	\addlegendentry{Exact orbit}
	
	\end{axis}
	\end{tikzpicture}
	
	\caption{
		Exact return orbit for
		$d(\theta)=1.3+0.36\cos(3\theta)$
		with initial condition $\theta_0=0.5$.
		The signed angular iterates alternate between the two sides of the
		former fixed point $\theta=0$ and approach the numerically resolved
		stable period-two cycle
		$\theta_{\pm}=\pm0.408120646\ldots$.
		The iterates are generated from the exact geometric return map,
		not from its first-order approximation.
	}
	\label{fig:stable-period-two-orbit}
	
\end{figure}
\section{Discussion and Open Problems}
\label{sec:discussion}

The purpose of this paper has been to pass from the local geometric
structure of the return map established in
\cite{barkatou2026return} to a genuinely global dynamical theory.
The main mechanism is the interaction between the thickness gradient,
the curvature-dependent preconditioner
\[
B_c
=
d(c)\bigl(I-d(c)S_c\bigr)^{-1},
\]
and the preconditioned Hessian
\[
K_c
=
B_c^{1/2}H_cB_c^{1/2}.
\]

Under a uniform local-return regime and a quantitative ascent
condition, the thickness itself becomes a strict Lyapunov quantity:
\[
d(F(c))-d(c)
\geq
\gamma
|\nabla_{\partial C}d(c)|^2.
\]
This yields square summability of the thickness gradient, vanishing
step size, convergence toward the critical set, and, under isolation
of the critical points, convergence of every orbit to a single
equilibrium.

The exact unit-circle computations of
Section~\ref{sec:circle-dynamics} also show that these hypotheses
describe a genuine dynamical regime rather than merely a technical
proof device. Once the relevant spectral threshold is crossed,
the numerical exact-map diagnostics show loss of thickness
monotonicity while the first-intersection branch remains globally real,
and a stable nontrivial period-two orbit is numerically resolved.

The remainder of this section discusses the scope of the theory and
several questions that remain open.

\subsection{The role of the preconditioned Hessian}
\label{subsec:discussion-preconditioned-hessian}

A central structural feature of the analysis is that the same
operator
\begin{equation}
K_c
=
B_c^{1/2}H_cB_c^{1/2}
\label{eq:discussion-K}
\end{equation}
appears at both the local and global levels.

At a critical point $c_*$, the exact derivative is
\begin{equation}
DF(c_*)
=
I+B_*H_*,
\label{eq:discussion-linearization}
\end{equation}
and the multipliers are
\begin{equation}
\lambda_j
=
1+\mu_j,
\label{eq:discussion-multipliers}
\end{equation}
where the $\mu_j$ are the eigenvalues of $K_*$. Thus a negative
preconditioned eigenvalue crosses the boundary of linear attraction
precisely when
\begin{equation}
\mu_j=-2.
\label{eq:discussion-local-minus-two}
\end{equation}

Globally, the second-order variation of the thickness has the form
\begin{equation}
d(F(c))-d(c)
=
\left\langle
z_c,
\left(
I+\frac12K_c
\right)z_c
\right\rangle
+
\mathcal E(c),
\label{eq:discussion-global-quadratic-form}
\end{equation}
where
\[
z_c
=
B_c^{1/2}\nabla_{\partial C}d(c),
\qquad
|\mathcal E(c)|
\leq
\eta(|\nabla_{\partial C}d(c)|)
|\nabla_{\partial C}d(c)|^2,
\]
with $\eta(s)\to0$ as $s\downarrow0$.
Accordingly, $-2$ is the spectral boundary at which the full
quadratic form $I+\frac12K_c$ ceases to be positive definite. This
statement concerns the full spectrum; the dynamically sampled
quadratic form along $z_c$ may remain positive even when another
eigendirection has crossed that boundary.

This coincidence is not imposed externally. It arises from the
geometry of the return map. Locally, $-2$ corresponds to the
multiplier $-1$; globally, it marks loss of positive definiteness of
the full second-order ascent form.

The one-dimensional exact example of
Section~\ref{subsec:circle-threshold-crossing} makes this relation
particularly transparent.

\subsection{Directional versus spectral ascent conditions}
\label{subsec:directional-versus-spectral}

The global theory was formulated using the directional condition
\begin{equation}
\left\langle
z_c,
\left(
I+\frac12K_c
\right)z_c
\right\rangle
\geq
\alpha_0|z_c|^2,
\label{eq:discussion-directional-condition}
\end{equation}
where
\[
z_c=B_c^{1/2}\nabla_{\partial C}d(c).
\]
The stronger spectral condition
\begin{equation}
\lambda_{\min}(K_c)
\geq
-2+\delta
\label{eq:discussion-spectral-condition}
\end{equation}
was introduced only as a convenient sufficient criterion.

This distinction is important. The return dynamics samples the
quadratic form only in the dynamically selected direction $z_c$.
There may therefore exist geometries for which
\[
\lambda_{\min}(K_c)<-2
\]
in a direction never sufficiently aligned with $z_c$, while the
actual return dynamics still possesses a strict Lyapunov structure.

Determining the optimal condition for global ascent is therefore a
more subtle problem than bounding the full spectrum of $K_c$.

\begin{openproblem}[Sharp global ascent criterion]
	\label{op:sharp-ascent}
	Characterize the largest class of geometries for which
	\begin{equation}
	d(F(c))>d(c)
	\qquad
	\text{whenever }
	\nabla_{\partial C}d(c)\neq0.
	\label{eq:op-exact-monotonicity}
	\end{equation}
	In particular, determine whether the directional quadratic condition
	\eqref{eq:discussion-directional-condition}, combined with an
	appropriate exact control of the nonlinear return geometry, can
	replace the stronger uniform spectral bound
	\eqref{eq:discussion-spectral-condition}.
\end{openproblem}

\subsection{How much smallness is really necessary?}
\label{subsec:discussion-smallness}

The uniform global one-step estimate was derived under the small-slope
condition
\begin{equation}
\|w\|_{L^\infty(\partial C)}
\leq
\varepsilon_0,
\qquad
w_c
=
\bigl(I-d(c)S_c\bigr)^{-1}
\nabla_{\partial C}d(c).
\label{eq:discussion-small-slope}
\end{equation}
Its role was geometric: it guarantees that every return point
$F(c)$ lies in the uniform normal-graph neighbourhood of its starting
point, allowing the exact closure relation to be converted into a
uniform quadratic remainder estimate.

The condition is sufficient, but no optimality is claimed. It may
be possible to obtain global control without requiring the entire
return to remain inside one normal coordinate patch centered at
$c$.

Indeed, the exact return map exists under considerably broader
geometric conditions than those used to prove the perturbative
one-step estimate. This is already visible in the circle experiments,
where the exact map remains real and single valued even after the
strict monotonicity regime has been lost.

\begin{openproblem}[Beyond the small-slope regime]
	\label{op:beyond-small-slope}
	Determine whether the exact return geometry admits a nonperturbative
	estimate for
	\[
	d(F(c))-d(c)
	\]
	that does not require the uniform small-slope hypothesis
	\eqref{eq:discussion-small-slope}.
	
	A particularly interesting question is whether the closure equation
	for the convex core can be combined with global convexity to produce
	an exact monotonicity criterion expressed directly in terms of
	$d$, $S$, and the first-return geometry.
\end{openproblem}

\subsection{The \texorpdfstring{$-2$}{-2} threshold and bifurcation}
\label{subsec:discussion-bifurcation}

The exact unit-circle family
\[
d(\theta)
=
d_0+\varepsilon\cos(m\theta)
\]
provides a tractable model for studying the loss of monotonicity.

At a thickness maximum, the preconditioned Hessian is
\begin{equation}
\mu
=
-
m^2\varepsilon
\frac{
	d_0+\varepsilon
}{
	1+d_0+\varepsilon
},
\label{eq:discussion-circle-mu}
\end{equation}
and the fixed-point multiplier is
\begin{equation}
\lambda
=
1+\mu.
\label{eq:discussion-circle-lambda}
\end{equation}
Thus
\[
\mu=-2
\]
is exactly the parameter value at which
\[
\lambda=-1.
\]

Section~\ref{subsec:circle-period-two} studies an exact-map example
with
\[
\mu<-2
\]
for which the former attracting maximum is unstable and a stable
period-two orbit is numerically resolved. This strongly suggests that the
$-2$ threshold is a genuine bifurcation boundary for at least some
one-parameter families of return geometries.

The numerical evidence, however, does not by itself establish the
type of bifurcation.

\begin{openproblem}[Local flip bifurcation]
	\label{op:flip-bifurcation}
	Consider a smooth one-parameter family of admissible thickness
	functions
	\[
	d_\eta
	\]
	for which a nondegenerate maximum persists as a fixed point
	$c_*(\eta)$ and one multiplier crosses $-1$ transversely.
	
	Determine geometric nondegeneracy conditions under which the return
	map undergoes a genuine period-doubling, or flip, bifurcation
	\cite{guckenheimer1983} at the
	parameter value satisfying
	\[
	\mu(\eta_*)=-2.
	\]
	In particular, derive the relevant normal-form coefficient directly
	from the geometry of $C$ and the derivatives of the thickness
	function.
\end{openproblem}

Such a result would provide an analytical explanation for the
period-two orbit found in the exact circle experiment.

\subsection{Dynamics beyond the first loss of monotonicity}
\label{subsec:beyond-monotonicity}

The numerically resolved stable period-two orbit indicates that once
the strict Lyapunov regime fails, the return dynamics need no longer
converge to fixed points. It does not imply that more complicated
dynamics must occur.

A first-order surrogate map is not sufficient for conclusions about
the exact return dynamics beyond the monotone regime. The appropriate
object for further study is the exact first-return map itself.

\begin{openproblem}[Higher-period dynamics]
	\label{op:higher-period}
	Within a smooth parameter family for which the exact return remains
	well defined, determine whether stable cycles of periods
	\[
	p\geq3
	\]
	can occur after the loss of the global ascent condition.
	
	For the unit-circle cosine family, construct the bifurcation diagram
	of the exact map
	\[
	T(\theta)
	=
	\theta+\delta(\theta)
	\pmod{2\pi}
	\]
	as $(d_0,\varepsilon,m)$ vary, using the exact first-intersection
	branch throughout.
\end{openproblem}

\begin{openproblem}[Complicated recurrent dynamics]
	\label{op:complex-recurrence}
	Determine whether there exist admissible domains in the class
	$\mathcal O_C$ for which the exact return map possesses invariant
	sets with genuinely nonperiodic recurrent dynamics.
	
	Any claim of chaotic behavior should be based on the exact return
	map and supported by appropriate dynamical evidence, such as
	positive Lyapunov exponents, entropy estimates, symbolic dynamics,
	or rigorously verified transverse homoclinic structure.
\end{openproblem}

We therefore make no chaos conjecture at the present stage.

\subsection{Nonisolated critical sets}
\label{subsec:nonisolated-critical}

The isolation assumption on
$\operatorname{Crit}(d)$ enters only in the final step of the global
convergence proof. Without it, the Lyapunov theory still yields
\begin{equation}
\operatorname{dist}
\bigl(
c_k,\operatorname{Crit}(d)
\bigr)
\longrightarrow0
\label{eq:discussion-critical-distance}
\end{equation}
and
\begin{equation}
\omega(c_0)
\subset
\operatorname{Crit}(d)
\cap
d^{-1}(d_\infty),
\label{eq:discussion-omega-critical-level}
\end{equation}
but convergence to one point need not follow from the present
argument.

\begin{openproblem}[Critical manifolds and degenerate equilibria]
	\label{op:nonisolated-critical}
	Suppose that $\operatorname{Crit}(d)$ contains a smooth critical
	submanifold or a more general compact connected component.
	
	Under what additional analytic or geometric assumptions does every
	return orbit still converge to a single point of that component?
	
	One possible direction is to seek a gradient inequality for the
	thickness near the critical set that strengthens square summability
	of the gradient to a finite-length property of the orbit.
\end{openproblem}

This question is distinct from the Morse-theoretic analysis of
isolated nondegenerate critical points.

\subsection{Basin boundaries and invariant manifolds}
\label{subsec:discussion-basin-boundaries}

Section~\ref{sec:basins-recurrence} established only the rigorous
inclusion
\begin{equation}
\partial\mathcal B(a)
\subset
\bigsqcup_{q\in\mathcal N}
\mathcal C(q)
\label{eq:discussion-boundary-inclusion}
\end{equation}
for an attracting equilibrium $a$.

A stronger description in terms of stable manifolds of saddle points
requires additional dynamical hypotheses and cannot be inferred from
the Lyapunov structure alone.

\begin{openproblem}[Separatrix structure]
	\label{op:separatrix-structure}
	Assume that all critical points are hyperbolic and that the return
	map is locally invertible at the relevant equilibria.
	
	Determine additional transversality conditions under which the
	boundaries of attraction basins are unions of stable manifolds of
	saddle points.
	
	More generally, identify conditions under which the resulting
	invariant-manifold structure is Morse--Smale.
\end{openproblem}

This issue is especially relevant in dimensions greater than one,
where different stable and unstable dimensions can interact in ways
that have no analogue for the circle map.

\subsection{Higher-dimensional return dynamics}
\label{subsec:higher-dimensional-dynamics}

The analytical results of Sections
\ref{sec:uniform-global-estimate}--\ref{sec:basins-recurrence}
are dimension independent. The exact circle computations, however,
probe only a one-dimensional phase space.

For
\[
\dim\partial C\geq2,
\]
the preconditioning operator $B_c$ becomes anisotropic and the
preconditioned Hessian $K_c$ may possess several competing spectral
directions. Consequently, loss of the ascent condition can occur in
only part of the tangent space.

\begin{openproblem}[Anisotropic loss of stability]
	\label{op:anisotropic-instability}
	Study parameter families for which one or several eigenvalues of
	$K_*$ cross $-2$ while the remaining eigenvalues stay inside
	$(-2,0)$.
	
	Determine how the geometry of the corresponding eigenspaces
	influences the bifurcation of the return dynamics and whether the
	first loss of stability produces periodic, quasiperiodic, or
	higher-dimensional invariant structures.
\end{openproblem}

Here the curvature of the convex core enters essentially: even for a
fixed Hessian $H_c$, changing the principal curvatures modifies the
operator $B_c$ and therefore changes the relevant preconditioned
spectrum.

\subsection{Robustness under perturbation}
\label{subsec:discussion-robustness}

The global ascent regime is expressed through strict inequalities.
This suggests a degree of robustness under sufficiently small smooth
perturbations of the geometry.

If, for example,
\[
\lambda_{\min}(K_c)
\geq
-2+\delta
\]
with a fixed margin $\delta>0$, then sufficiently small perturbations
of $d$ and of the geometry of $\partial C$ preserve the corresponding
spectral inequality.

A complete perturbation theory must nevertheless account
simultaneously for the radial map, the reciprocal first-hit map, the
return transversality constant, and the uniform local-return
threshold.

\begin{openproblem}[Structural robustness]
	\label{op:structural-robustness}
	Quantify the stability of the global convergence regime under
	perturbations of the pair $(C,\Omega)$.
	
	In particular, determine explicit geometric norms and perturbation
	bounds that preserve:
	\begin{enumerate}
		\item the single-valued return construction;
		\item uniform return transversality;
		\item the global one-step estimate;
		\item strict thickness ascent;
		\item convergence to isolated critical points.
	\end{enumerate}
\end{openproblem}

\subsection{Continuous-time interpretation}
\label{subsec:continuous-time-discussion}

The leading return displacement
\[
F(c)-c
\approx
B_c\nabla_{\partial C}d(c)
\]
has the formal appearance of one step of a
curvature-preconditioned gradient-ascent process.

The present paper does not identify the discrete return map with a
time discretization of a continuous flow. Such an interpretation
would require an explicit small-step scaling and a consistency
analysis over finite time intervals. This distinction was already
important in the foundational analysis
\cite{barkatou2026return}.

A natural candidate continuous dynamics is
\begin{equation}
\dot c
=
B_c\nabla_{\partial C}d(c),
\label{eq:formal-continuous-flow}
\end{equation}
for which
\begin{equation}
\frac{d}{dt}d(c(t))
=
\left\langle
\nabla_{\partial C}d(c(t)),
B_{c(t)}
\nabla_{\partial C}d(c(t))
\right\rangle
\geq0.
\label{eq:continuous-thickness-increase}
\end{equation}

\begin{openproblem}[Scaled continuous-time limit]
	\label{op:continuous-limit}
	Construct a geometrically natural parameter family for which the
	return displacement tends uniformly to zero and prove, after an
	appropriate rescaling of discrete time, convergence of the return
	orbits to solutions of
	\eqref{eq:formal-continuous-flow} or to the corresponding limiting
	preconditioned flow.
\end{openproblem}

\subsection{Perspective}
\label{subsec:discussion-perspective}

The results of this paper identify a specific global regime of the
return dynamics:
\[
\text{uniform return control}
\quad+\quad
\text{preconditioned ascent}
\quad+\quad
\text{isolated critical points}
\]
implies
\[
\text{strict Lyapunov structure}
\quad+\quad
\text{global convergence}
\quad+\quad
\text{absence of nontrivial recurrence}.
\]

Within this regime, the thickness landscape organizes the
asymptotic dynamics, but its action is mediated by the curvature of
the convex core through $B_c$. Outside this regime, the geometry can
generate qualitatively different behavior even while the return map
remains regular and single valued.

The numerically resolved period-two orbit of the exact unit-circle map
marks the first such transition exhibited here. It shows that the boundary between
convergent and nonconvergent dynamics is itself a geometric object
worthy of study.

The resulting picture is therefore not one of an abstract gradient
algorithm imposed on the boundary. The discrete dynamics is produced
by the nested geometry itself, and both its convergent regime and its
possible bifurcations are encoded by the interaction between
thickness variation, curvature, and normal return geometry.

\section*{Acknowledgments}


\end{document}